\documentclass[11pt,a4paper]{article}
\usepackage{amssymb,latexsym}
 \usepackage[latin1]{inputenc}
\usepackage{a4}
\usepackage{bm, amsmath, amssymb, amsthm}
\usepackage{graphicx}

\def\<{\langle}
\def\>{\rangle}
\def\eps{\varepsilon}
\def\RR{\mathbb{R}}

\def\calf{\mathcal{F}}
\def\tr{\operatorname{Tr\,}}
\def\id{\operatorname{id\,}}
\def\Div{\operatorname{div}}

\def\Ric{\operatorname{Ric}}
\def\vol{\operatorname{vol}}
\newcommand{\Sm}{{\mbox{\rm S}}}

\def\eq{\hspace*{-1.5mm}&=&\hspace*{-1.5mm}}
\def\plus{\hspace*{-1.5mm}&+&\hspace*{-1.5mm}}
\def\minus{\hspace*{-1.5mm}&-&\hspace*{-1.5mm}}
\def\dt{\partial_t}

\newtheorem{thm}{Theorem}[section]
\newtheorem{cor}[thm]{Corollary}
\newtheorem{lem}[thm]{Lemma}
\newtheorem{prop}[thm]{Proposition}
\newtheorem{example}[thm]{Example}

\newtheorem{rem}{Remark}[section]

\title{The mixed Einstein-Hilbert action and extrinsic geometry of foliated manifolds}

\author{
       Elisabetta Barletta\footnotemark[1],\quad
       Sorin Dragomir\footnote{Dipartimento di Matematica, Informatica ed Economia,
       Universit\`{a} degli Studi della Basilicata,
       \newline
       e-mail: {\tt elisabetta.barletta@unibas.it}, {\tt sorin.dragomir@unibas.it}},\ \
       \ and \
       Vladimir Rovenski\footnote{Mathematical Department, University of Haifa, e-mail: {\tt rovenski@math.haifa.ac.il}}
       }

\begin{document}

\date{}

\maketitle

\begin{abstract}
We develop variation formulas for the quantities of extrinsic geometry
for adapted variations of metrics on almost-product (e.g. foliated) Riemannian manifolds,
and apply them to study the total mixed scalar curvature of a distribution
-- analogue of the classical Einstein-Hilbert action.
The mixed scalar curvature $\Sm_{\,\rm mix}$ is the averaged sectional curvature
over all planes that contain
vectors from both distributions of an almost-product structure
and the variations we consider preserve orthogonality of the distributions.
We derive the directional derivative $D J_{\,\rm mix}$ (of the total $\Sm_{\,\rm mix}$) for adapted variations of metrics
on closed almost-product manifolds and foliations of arbitrary dimension.
The obtained Euler-Lagrange equations are presented in two equiva\-lent forms:
in terms of extrinsic geometry and intrinsically using the partial Ricci tensor.
Certainly, these mixed field equations admit amount of solutions (e.g., twisted products).
\end{abstract}



\section*{Introduction}\label{sec:intro}

Foliations (i.e., decompositions of manifolds into collections of submanifolds, cf.~\cite{cc1})
grew out of the theory of dynamical systems; many models in mechanics and relativity are foliated
(e.g., warped products).
Riemannian geometry of foliations is well developed since years and has different aspects, local and global, intrinsic and extrinsic, see survey in \cite{rov-m,rw-m}.
\textit{Extrinsic geometry} of a foliation describes the properties depending on the second fundamental form
of the leaves and its invariants. A Riemannian manifold may admit many kinds of geometrically interesting foliations: totally geodesic and Riemannian foliations are most popular examples.

The {problem of minimizing geometric quantities} has been very popular since
long time: recall, for example, classical isoperimetric inequalities, Fenchel estimates of total curvature of curves and estimates of total mean curvatures of compact submanifolds, see \cite{lr1996}.
In the context of foliations, Gluck and Ziller in 1986 considered the problem of
minimizing functions like volume, total energy and bending defined for $k$-plane fields
on Riemannian manifolds, one also has \cite{lp1993,lw2008,rw2008}.
 In~all the cases mentioned above, they consider a fixed Riemannian manifold $(M,g)$
and look for geometric objects (curves, hypersurfaces, foliations) minimizing geometric
quantities defined usually as integrals of curvatures of different types.

\label{R-EHa}\rm
The equations of mechanics can be also obtained as solutions of a~variational problem, using a~suitable functional, called the action on the configuration space. The {Einstein--Hilbert action} in general relativity
yields the Einstein's field equations through the principle of least action.
The gravitational part of the action is $J(g) = \frac1{k}\int_M \Sm(g)\,{\rm d}\vol_g$,
where $g$ is the metric of index~$1$,
$\Sm(g)$ is the \textit{scalar curvature} of the spacetime~$(M^4,g)$, $k=16\,\pi G/c^4$,
$G$ is the gravitational constant and $c$ is the speed of light in vacuum.
The integral is taken over the whole spacetime if it converges;
otherwise, a modified definition of $J(g)$ where one integrates over arbitrarily large, relatively compact domain $\Omega$, still yields the Einstein equations.
The equations in the presence of matter are given by adding the matter action to the Hilbert-Einstein action.

On the other hand, there is definite interest on prescribing geometric quantities of given objects
(see \cite{aub}, and for foliations, see \cite{cw2005,r2010,rz2012,sw}):
given a manifold $M$ and a geometric quantity $Q$ (function, vector or tensor field) one may search for
a Riemannian metric $g$ on $M$ for which a given geometric invariant (say, curvature of some sort)
coincides with $Q$.
 One of the principal problems of extrinsic geometry of foliations reads as follows:
 \textit{Given a foliated manifold $(M,\calf)$ and an extrinsic geometric property $($P$)$ of a submanifold,
 find a~Riemannian metric $g$ on $M$ such that the leaves of $\calf$ enjoy $($P$)$ with respect to~$g$}.
 For example, there exist metrics making a Reeb foliation (on the three-dimensional sphere) totally umbilical
but there is no metric making a Reeb foliation harmonic (i.e., the leaves become minimal submanifolds).

In \cite{rw-m} the new approach from the two we just described is presented to problems
in extrinsic geometry of codimension one foliations: given a foliated manifold $(M, \calf)$ and
a property $Q$ of a submanifold, depending on the principal curvatures of the leaves,
study Riemannian metrics, which minimize the integral of $Q$ in the class of $\calf$-truncated metrics
(i.e., the unit vector field $N$ orthogonal to $\calf$ is the same for all metrics of the variation family).
Certainly (like in some of the cases mentioned before) such Riemannian structures need not exist, but if they do, they usually have interesting geometric properties.

Let ${\rm Sym}^2(M)$ be the space of all symmetric $(0,2)$-tensors fields tangent to~$M$.
A semi-Riemannian (pseudo-Riemannian) metric of index $q$ on $M$ is an element $g\in{\rm Sym}^2(M)$
 such that each $g_x$ is a non-degenerate bilinear form of index $q$ on the tangent space $T_xM$ for all $x\in M$.
When $q=0$, that is $g_x$ is positive definite for all $x\in M$, we say that $g$ is a Riemannian metric (resp. Lorentz metric, when $q=1$).
If a distribution $\widetilde{\mathcal D}$ (i.e., a vector subbundle of the tangent bundle $TM$) is given, then a complementary distribution ${\mathcal D}$ to $\widetilde{\mathcal D}$ in $TM$ can be obtained. Indeed, since $M$ is paracompact and of class $C^\infty$, there exists a semi-Riemannian metric $g$ of class $C^\infty$.
Then we can take ${\mathcal D}$ as the orthogonal distribution to $\widetilde{\mathcal D}$ with respect to this metric. If $\widetilde{\mathcal D}$ is nondegenerate (i.e., $\widetilde{\mathcal D}_x$ is a nondegenerate subspace of the semi-Euclidean space $(T_x M,\, g_x)$ for every $x\in M$), then ${\mathcal D}$ is also nondegenerate.
Thus, we are entitled to consider a connected manifold $M^{n+p}$ with a semi-Riemannian metric $g$ and a pair
of complementary orthogonal nondegenerate distributions $\widetilde{\mathcal D}$ and ${\mathcal D}$ of ranks $\dim_{\,\RR}\widetilde{\mathcal D}_x=n$ and $\dim_{\,\RR}{\mathcal D}_x=p$ for every $x\in M$;
this is called an {almost-product structure} on $M$.
 A.Gray and B.O'Neill calculated the curvatures of the distributions from the curvature of $M$,
using configuration tensors
(these are analogs of the second fundamental form of a submanifold).
 Extending the definition, we shall say that \textit{extrinsic geometry}
of an almost-product structure describes the properties, which can be expressed using
configuration tensors (i.e., the integrability tensors and the second fundamental forms).
In~\cite{bf}, a~tensor calculus, adapted to decomposition
\begin{equation}\label{e:Decomposition}
 TM = \widetilde{\mathcal D} \oplus {\mathcal D},
\end{equation}
is developed to study the geometry of both the distributions and the ambient manifolds.
The~sectional curvature $K(X, Y)$, where $X\in{\mathcal D}_\calf$ and $Y\in{\mathcal D}$, is called \textit{mixed}.
The ``mixed components" of the curvature tensor (involved in the Jacobi equation) regulate the deviation of leaves along the leaf geodesics.
The \textit{mixed scalar curvature} $\Sm_{\,\rm mix}$ is the averaged mixed sectional curvature.
In relativity, it measures the relative acceleration of two particles moving forward on neighboring geodesics.

Our main objective is to develop variation formulas for the quantities of extrinsic geometry
for \textit{adapted} \textit{variations} of metrics on almost-product (e.g. foliated) Riemannian manifolds,
and to apply them to study to study Riemannian structures on a closed manifold $M$, minimizing integral
of the quantity $Q=\Sm_{\,\rm mix}$ for adapted (i.e., preserving the decomposition \eqref{e:Decomposition})
variations of metrics $g_t\ (g_0=g,\ |t|<\eps)$, to deduce the Euler--Lagrange equations and
characterize the critical metrics in certain distinguished classes of almost-product structures.
The cases of an open manifold $M$ and general variations of metrics will be studied in further~works.

This functional, called the \textit{mixed Einstein-Hilbert action}, is imitative of Einstein-Hilbert one,
except that the scalar curvature is replaced by the mixed scalar curvature.
Note that the quantity $Q=\Sm_{\,\rm mix}$ is built of the invariants of extrinsic geometry of both distributions, see~\cite{wa1}; hence, our study belongs to the extrinsic geometry of foliations.
The mixed Einstein-Hilbert action for a globally hyperbolic spacetime $(M^4,\,g)$ equipped with a~Bernal-S\'{a}nchez foliation was considered in~\cite{bdrs}, where Euler-Lagrangue equations (called the mixed gravitational field equations) were deduced using variation formulas for the curvature and Ricci tensor,
then their linea\-rization and solution for empty space were derived.

 Our approach
 is based on the variation formulas for the extrinsic geometry of foliations of arbitrary (co)dimension; thus, the paper shows some progress in methods of~\cite{rw-m}, where variation formulas and functionals were studied for codimension one foliations.
As~we shall see shortly, the Euler-Lagrange equations for the mixed Einstein-Hilbert action
involve several new tensors and a new type of Ricci curvature (introduced in~\cite{r2010}),
whose properties need to be further investigated.
The~questions for further study are: stability conditions,
and extrinsic geometry of critical metrics with respect to adapted (and general) variations,~etc.

The paper contains introduction and two chapters.

Section~\ref{sec:mixed-action} develops variation formulas for the quantities of extrinsic geometry
for {adapted} {variations} of metrics on Riemannian almost-product manifolds,
and applies them to study the mixed Einstein-Hilbert action on such manifolds.
Section~\ref{sec:main} contains applications of the above to foliated manifolds
(i.e., ${\mathcal D}$ is the normal subbundle to a foliation $\calf$).
The main goals are the Euler-Lagrange equations for almost-product structure (Proposition~\ref{T-main00})
and for foliations (Theorem~\ref{T-main01} and corollaries).
We look at several classes of foliations: characterize critical metrics and give examples.
 Throughout the paper everything (manifolds, distributions, foliations, etc.) is assumed to be
smooth (i.e., $C^{\infty}$-differentiable) and oriented.

\section{Mixed Einstein-Hilbert action on almost-product manifolds}\label{sec:mixed-action}

Using the natural representation of ${\rm O}(p)\times {\rm O}(n)$ on $TM$,
A.M. Naveira obtained (cf. \cite{N1983}) thirty-six distinguished classes of Riemannian almost-product manifolds
$(M,g,\widetilde{\mathcal D},{\mathcal D})$; half of them are foliated.
Following this line of research, several geometers completed the geometric interpretation, gave examples for each class, and studied the different classes of almost-product structures.
The~following convention is adopted for the range of indices
\[
 a,b,\ldots\in\{1,\ldots, n\},\quad
 i,j\,,\ldots\in\{1,\ldots, p\}.
\]

\subsection{The total mixed scalar curvature and the space of adapted metrics}\label{subsec:Smix}

The~\textit{mixed scalar curvature} function on $M$ is defined as (cf. \cite{rov-m, wa1})
\begin{equation}\label{eq-wal2}
 \Sm_{\,\rm mix} =\sum\nolimits_{\,a,i} 
 K(E_a, {\mathcal E}_{i})
 =\sum\nolimits_{\,a,i} 
 g(R^\nabla(E_a, {\mathcal E}_{i})E_a,\, {\mathcal E}_{i})\,,
\end{equation}
where 
 $\{E_a,\,{\mathcal E}_{i}\}$ is a local $g$-orthonormal frame on $TM$
such that $\{E_{a}\}$ are tangent to $\widetilde{\mathcal D}$.
Here, $R^\nabla(X,Y)=\nabla_Y\nabla_X-\nabla_X\nabla_Y+\nabla_{[X,Y]}$ is the~curvature tensor
of the Levi-Civita connection $\nabla:TM\times C^\infty(TM)\to C^\infty(TM)$.
 The central object in the paper is the total mixed scalar curvature,
referred to as the \textit{mixed Einstein-Hilbert action} and given by
\begin{equation}\label{E-Jmix}
 J_{\,\rm mix}(g) = \int_{M} \Sm_{\,\rm mix}(g)\, {\rm d} \vol_g.
\end{equation}
If the integral \eqref{E-Jmix} does not converge, one may use the modified definition,
where the region of integration is a ``large" relatively compact domain $\Omega\subset M$.
 Note that $\Sm_{\,\rm mix}$ is the Gaussian curvature for surfaces foliated by curves,
 and $\Sm_{\,\rm mix} = \Ric(N,N)$ for a~codimensi\-on-one foliation with a unit normal $N$.
 In the last case, action \eqref{E-Jmix} is given by (see also Sect.~\ref{subsec:codim1fol})
\begin{equation}\label{E-Jmix-N}
 J_{\,\rm mix}(g) = \int_{M} \Ric_{\,g}(N,N)\,{\rm d}\vol_g\,,
\end{equation}
where $\Ric_{\,g}(X,Y)=\tr g(R^\nabla(X,\,\cdot\,)Y,\,\cdot\,)$ is the Ricci tensor of $(M,g)$.

Our considerations through the present paper are confined to the Riemannian case ($g \in {\rm Riem} (M)$) and we relegate the arbitrary (and Lorentzian) signature case to further work.

Let $\mathfrak{X}_M$ be the module over $C^\infty(M)$ of all vector fields on $M$ (i.e., sections of the tangent bundle $TM$), and $\mathfrak{X}_{\mathcal D}$ (resp. $\mathfrak{X}_{\widetilde{\mathcal D}}$) the module over $C^\infty(M)$ of all vector fields in ${\mathcal D}$ (resp. $\widetilde{\mathcal D}$).
For every $X\in{\mathfrak X}_M$, let $\widetilde{X} \equiv X^\top$ be the component of $X$ along $\widetilde{\mathcal D}$ (resp. $X^\perp$ the component of $X$ along ${\mathcal D}$) with respect to the direct sum decomposition (\ref{e:Decomposition}).
A~tensor field $S \in {\rm Sym}^2(M)$ is said to be {\em adapted} if $S(\tilde X, Y^\perp) = 0$ for any $X,Y\in{\mathfrak X}_M$. Let ${\mathfrak M}\equiv{\mathfrak M}(\widetilde{\mathcal D},\,{\mathcal D})$ consist of all adapted symmetric tensor fields on $(M,\,\widetilde{\mathcal D},\,{\mathcal D})$.
The domain of $J_{\,\rm mix}$ is {\em a priori} the space
${\rm Riem}(M,\,\widetilde{\mathcal D},\,{\mathcal D})\equiv {\rm Riem}(M) \cap {\mathfrak M}$
of all adapted metrics, i.e.,
$J_{\,\rm mix} : {\rm Riem} (M,\,\widetilde{\mathcal D},\,{\mathcal D})\to{\mathbb R}$.

We say that a tensor $S \in {\rm Sym}^2(M)$ is $\widetilde{\mathcal D}$-\emph{truncated}
if $X^\perp\,\rfloor\,S=0$ (resp. ${\mathcal D}$-\emph{truncated} if $\tilde{X}\,\rfloor\,S=0$) for any $X\in{\mathfrak X}_M$. This notion is extended for $(1,1)$-tensors.
Let ${\mathfrak M}_{\widetilde{\mathcal D}}$ and ${\mathfrak M}_{\mathcal D}$ be, respectively, the spaces of $\widetilde{\mathcal D}$-truncated and $\mathcal D$-truncated symmetric $(0,2)$-tensor fields. Then ${\mathfrak M}_{\mathcal D}$
and ${\mathfrak M}_{\widetilde{\mathcal D}}$ are subspaces of $\mathfrak M$ and
\begin{equation}\label{e:Decomposition2}
 {\mathfrak M} = {\mathfrak M}_{\widetilde{\mathcal D}} \oplus {\mathfrak M}_{\mathcal D}\,,
\end{equation}
the decomposition is orthogonal with respect to the inner product $g^\ast$ induced on $\mathfrak M$
by each $g \in {\rm Riem}(M,\,\widetilde{\mathcal D},\, {\mathcal D})$.
 A~tensor $S \in {\mathfrak M}_{\mathcal D}$ is
${\mathcal D}$-\emph{conformal} if $S=s\,g^\perp$ for some $s \in C^\infty (M, \RR )$,
in particular, ${\mathcal D}$-\emph{scaling} if $s$ is constant.
Given $g \in {\rm Riem}(M,\,\widetilde{\mathcal D},\, {\mathcal D})$, the subspace
of ${\mathfrak M}$, consisting of biconformal adapted tensors, splits into direct sum of
${\mathcal D}$- and $\widetilde{\mathcal D}$-conformal components, i.e.,
${\mathfrak B} = {\mathfrak B}_{\widetilde{\mathcal D}} \oplus {\mathfrak B}_{\mathcal D}$.
Certainly, the space ${\mathfrak B}$ is wider than the space of all conformal to $g$ metrics.

For each $(0,2)$-tensor field $S$ tangent to $M$
we define its truncated components $\widetilde{S}, \, S^\bot \in \Gamma(T^\ast M \otimes T^\ast M)$ by setting
$\widetilde S(X,Y) = S(\tilde X,\tilde Y)$ and $S^\perp(X,Y) = S(X^\perp, Y^\perp)$ for any $X, Y\in{\mathfrak X}_M$.
If $S \in {\rm Sym}^2(M)$ then $S\in{\mathfrak M} \Longleftrightarrow S = S^\bot + \widetilde{S}$,
see \eqref{e:Decomposition2}.
 In particular, for any adapted metric $g\in{\rm Riem}(M,\,\widetilde{\mathcal D},\,{\mathcal D})
\equiv {\rm Riem}(M) \cap {\mathfrak M}$ one has a decomposition
\[
 g = g^\perp + \tilde{g}\quad \Longleftrightarrow \quad
 g=\bigg(\begin{array}{cc}
   g^\perp_{\,|\,{\mathcal D}} & 0 \\
   0 & \tilde{g}_{\,|\,\widetilde{\mathcal D}}
 \end{array}\bigg) .
\]
Our purpose in this paper is to compute the directional derivatives
\begin{equation}\label{E-DJ}
 D_g J_{\,\rm mix} : T_g{\rm Riem}(M,\,\widetilde{\mathcal D},\,{\mathcal D}) \equiv {\mathfrak M}\ \to\ {\mathbb R}
\end{equation}
for any $g \in {\rm Riem}(M,\,\widetilde{\mathcal D},\,{\mathcal D})$ on almost-product
or foliated manifolds $(M,\,\mathcal D,\,\widetilde{\mathcal D})$
and study the extrinsic geometry of $\widetilde{\mathcal D}$ and ${\mathcal D}$ in $(M, g)$,
where $g$ is a critical point of $J_{\,\rm mix}$ with respect to adapted variations \eqref{E-Sdtg}.
Certainly, we can restrict ourselves to the cases
$D_g J_{\,\rm mix}:{\mathfrak M}_{\mathcal D}\to{\mathbb R}$
or $D_g J_{\,\rm mix}:{\mathfrak M}_{\widetilde{\mathcal D}}\to{\mathbb R}$,
when $g$ is either a $\mathcal D$-critical point
(i.e., $D_g J_{\,\rm mix}(S)=0$ for every $S\in{\mathfrak M}_{\mathcal D}$) or a $\widetilde{\mathcal D}$-critical
point (i.e., $D_g J_{\,\rm mix}(S)=0$ for every $S\in{\mathfrak M}_{\widetilde{\mathcal D}}$) of $J_{\,\rm mix}$.

The so called musical isomorphisms $\sharp$ and $\flat$ will be used for arbitrary $(k,l)$-tensor fields,
which form the infinite-dimensional vector spaces $T^k_l M$ over $\RR$ and modules over $C^\infty(M)$.
For~example, if $\omega \in T^1_0 M$ is 1-form and $X \in {\mathfrak X}_M$ then
$\omega(Y)=g(\omega^\sharp,Y)$ and $X^\flat(Y) =g(X,Y)$ for any $Y\in {\mathfrak X}_M$,
and if $S\in{\rm Sym}^2(M)$ then the $(1,1)$-tensor field $S^\sharp \in T^1_1 M$ is
$g(S^\sharp X, Y) = S(X,Y)$ for any $X,Y\in {\mathfrak X}_M$.
Note that if $S\in{\mathfrak M}$ then $\widetilde{\mathcal D}$ and ${\mathcal D}$ are $S^\sharp$-invariant.

\subsection{The fundamental tensors of almost-product manifolds}\label{subsec:PRtensor}

As we shall see shortly, the Euler-Lagrange equations of the variational principle associated to \eqref{E-DJ}
involve a new kind of Ricci curvature (previously introduced in \cite{r2010}, and studied in \cite{bdrs} for global hyperbolic spacetimes $(M^4,g)$), whose properties need to be further investigated.
The ${\mathcal D}$-truncated symmetric $(0,2)$-tensor field ${r}_{\,\mathcal D}$ on $(M,g)$ given by
\begin{equation}\label{E-Rictop2}
 {r}_{\,{\mathcal D}}(X,Y) = \sum\nolimits_{a} g(R^\nabla (E_a, \, X^\perp)E_a, \, Y^\perp), \qquad  X,Y\in {\mathfrak X}_M,
\end{equation}
is referred to as the \textit{partial Ricci tensor} concentrated on $\mathcal D$.
A $2$-dimensional subspace $\sigma \subset T_x M$ is {\em mixed} if it admits a linear basis $\{v, w\}\subset\sigma$ such that $v\in\widetilde{\mathcal D}_x$ and $w \in {\mathcal D}_x$.
Therefore, the \textit{partial Ricci} curvature in the direction of
a unit vector $X\in{\mathcal D}$ is the mean value of sectional curvatures over all mixed planes containing $X$.
Similarly, the partial Ricci tensor on $\widetilde{\mathcal D}$ is
\begin{equation}\label{E-Rictop}
 {r}_{\,\widetilde{\mathcal D}}(X,Y)
 =\sum\nolimits_{i} g(R^\nabla ({\mathcal E}_{i},\,\tilde X){\mathcal E}_{i},\, \tilde Y),
 \qquad X,Y\in {\mathfrak X}_M.
\end{equation}
In particular (by \eqref{eq-wal2}),
 $\tr_{g}{r}_{\,\widetilde{\mathcal D}}=\Sm_{\,\rm mix}=\tr_{g}{r}_{\,\mathcal D}$.
To study the partial Ricci curvature (e.g., in Proposition~\ref{L-CC-riccati} below) we introduce several tensors.
 Let~$T, h:\widetilde{\mathcal D}\times \widetilde{\mathcal D}\to{\mathcal D}$ and
$\tilde T, \tilde h:{\mathcal D}\times{\mathcal D}\to\widetilde{\mathcal D}$ be
the integrability tensors and the second fundamental forms
of $\widetilde{\mathcal D}$ and~${\mathcal D}$, respectively,
\begin{eqnarray*}
  T(X,Y)=(1/2)\,[X,\,Y]^\perp,\quad h(X,Y) \eq (1/2)\,(\nabla_X Y+\nabla_Y X)^\perp,\\
 \tilde T(X,Y)=(1/2)\,[X,\,Y]^\top,\quad \tilde h(X,Y) \eq (1/2)\,(\nabla_X Y+\nabla_Y X)^\top.
\end{eqnarray*}
The mean curvature vector fields of $\widetilde{\mathcal D}$ and ${\mathcal D}$
are $H=\tr_{g} h$ and $\tilde H=\tr_{g}\tilde h$, respectively.
 The~distribution $\widetilde{\mathcal D}$ (and similarly for ${\mathcal D}$)
is called \textit{totally umbilical}, \textit{harmonic}, or \textit{totally geodesic}, if
 $h=\frac1nH\,\tilde g,\ H =0$, or $h=0$, respectively.

The \textit{conullity tensors}
$\tilde C: {\mathfrak X}_{\widetilde{\mathcal D}}\times {\mathfrak X}_{\mathcal D}\to{\mathfrak X}_{\mathcal D}$
and $C: {\mathfrak X}_{\mathcal D}\times {\mathfrak X}_{\widetilde{\mathcal D}}\to {\mathfrak X}_{\widetilde{\mathcal D}}$
are defined~by
\begin{equation}\label{E-conulC}
 \tilde C_Y(X)=-(\nabla_{X}\,Y)^\perp,\qquad
 C_Z(W)=- \left( \nabla_{W}\,Z \right)^\top ,
\end{equation}
for any $ Y,W\in{\mathfrak X}_{\widetilde{\mathcal D}}$ and $X,Z\in{\mathfrak X}_{\mathcal D}$.
Let $\tilde A_Y$ be the Weingarten operator of ${\mathcal D}$ with respect to $Y$ (i.e., dual to $\tilde h$),
and similarly $A_Z$. The dual to $T$ and $\tilde T$ operators $T^\sharp_Z$ and $\tilde T^\sharp_Y$ are given by $g(T^\sharp_Z(X),Y)=g(T(X,Y),Z)$ and $g(\tilde T^\sharp_Y(X),Z)=g(\tilde T(X,Z),Y)$. Then
\begin{equation}\label{E-CA2}
\begin{cases} \tilde C_Y(X) = \tilde A_Y(X)+\tilde T^\sharp_Y(X), \;\; X\in{\mathfrak X}_{\mathcal D},
 \cr
 C_Z(W) = A_Z(W)+T^\sharp_Z(W), \;\; W\in{\mathfrak X}_{\widetilde{\mathcal D}},
 . \cr \end{cases}
\end{equation}
 The \textit{divergence of a vector field} $\xi\in{\mathfrak X}_M$ is given by
\[
 \Div\xi=\tr (\nabla \xi) =\sum\nolimits_{a} g(\nabla_{a} \,\xi, E_a)
 +\sum\nolimits_{i} g(\nabla_{i} \,\xi, {\mathcal E}_i).
\]
The Divergence Theorem is  $\int_{M} (\Div\xi)\,d\vol_g =0$, when $M$ is closed;
this is also if $M$ is open and $\xi$ is supported in a relatively compact domain $\Omega\subset M$.
The $\widetilde{\mathcal D}$-\textit{divergence} of $\xi$ is defined by
 $\widetilde{\Div}\,\xi=\sum\nolimits_{a} g(\nabla_{a}\xi, E_a)\,$.
Respectively, $\Delta f=\Div(\nabla f)$ is the Laplacian
and $\widetilde\Delta\,f=\widetilde{\Div}\,(\widetilde\nabla\,f)$
the $\widetilde{\mathcal D}$-\textit{Laplacian} on functions.
For $X\in{\mathfrak X}_{\widetilde{\mathcal D}}$ and a function $f\in C^2(M)$, we have, see \cite{rovwol},
\begin{eqnarray}\label{E-divN}
 \widetilde{\Div}\,X = \Div X+g(\tilde H, X),\qquad
 \widetilde\Delta\,f =\Div(\,\widetilde\nabla\,f +f\tilde H)-(\,\Div\tilde H)f.
\end{eqnarray}
Indeed, using $\tilde H=\sum\nolimits_{i\le p} \tilde h({\mathcal E}_{i}, {\mathcal E}_{i})$
and $g(X,\,{\mathcal E}_{i})=0$, one derives~(\ref{E-divN})$_1$
\begin{equation*}
 \Div X-\widetilde{\Div}\,X =\sum\nolimits_{i} g(\nabla_{{\mathcal E}_{i}} X,\,{\mathcal E}_{i})
 =-\sum\nolimits_{i} g(h({\mathcal E}_{i}, {\mathcal E}_{i}), X) = -g(\tilde H, X).
\end{equation*}
The identity $\Div(\phi\,\xi)=\phi\Div\xi + \xi(\phi)$
together with (\ref{E-divN})$_1$ for $X=\widetilde\nabla f$ imply (\ref{E-divN})$_2$.

 Define $(1,1)$-tensors, called the \textit{Casorati operators} of ${\cal D}$ (and similarly for $\widetilde{\cal D}$):
 \[
 {\mathcal A}:=\sum\nolimits_{\,a} A_{a}^2,\quad
 {\mathcal T}:=\sum\nolimits_{\,a}(T_{a}^\sharp)^2.
\]
For $\widetilde{\mathcal D}$-valued  $(1,2)$-tensors $P$,
define $(\widetilde{\Div}\,P)(X,Y)=\sum\nolimits_a g((\nabla_a\,P)(X,Y), {E}_a)$.
We have
\begin{eqnarray}\label{E-divP}
 \widetilde{\Div}\,P = \Div P+\<\tilde H, P\>\,,
\end{eqnarray}
where $\<P,\,H\>(X,Y):=g(P(X,Y),\,H)$ is a $(0,2)$-tensor.
For example, $\widetilde{\Div}\,\tilde h = \Div\tilde h+\<\tilde H,\tilde h\>$.
Similarly we obtain $\Div^\perp h = \Div h+\<H, h\>$.

 The \textit{deformation tensor} ${\rm Def}_{\mathcal D}\,H$ of $H$ is the symmetric part of $\nabla H$
restricted to~${\mathcal D}$, i.e.,
\[
 2\,{\rm Def}_{\mathcal D}\,H(X,Y)=g(\nabla_X H, Y) +g(\nabla_Y H, X),\quad X,Y\in {\mathfrak X}_{\mathcal D}.
\]
(We call $H$ a ${\mathcal D}$-\textit{Killing vector field} if ${\rm Def}_{\mathcal D}\,H=0$).
It is also useful to identify the antisymmetric part of $\nabla H$ restricted to~${\mathcal D}$,
which is regarded as a 2-form $d_{\mathcal D}\,H$,
\[
 2\,d_{\mathcal D}\,H(X,Y)=g(\nabla_X H, Y)-g(\nabla_Y H, X),\quad X,Y\in{\mathfrak X}_{\mathcal D}.
\]
Using \eqref{E-CA2} and $\tr_{g}(A_Y T^\sharp_X)=0$,
define the ${\mathcal D}$-truncated
$(0,2)$-tensor $\Psi$ by the identity
\begin{equation}\label{E-Psi}
 \Psi(X,Y) =\tr_{g}(C_Y C_X) =\tr_{g}(A_Y A_X+\!T^\sharp_Y T^\sharp_X),
 \quad X,Y\in{\mathfrak X}_{\mathcal D}\,.
\end{equation}
Similarly we define the tensor
\[
 \widetilde\Psi(X,Y)=\tr_{g}(\tilde A_Y \tilde A_X +\tilde T^\sharp_Y \tilde T^\sharp_X),
 \quad X,Y\in{\mathfrak X}_{\widetilde{\mathcal D}}.
\]

\begin{prop}\label{L-CC-riccati}
Let $g\in{\rm Riem}(M,\,\widetilde{\mathcal D},\,{\mathcal D})$. Then
\begin{eqnarray}\label{E-genRicN}
 r_{\,{\mathcal D}} -\Div\tilde h -\<\tilde H,\tilde h\>
  +\widetilde{\mathcal A}^\flat +\widetilde{\mathcal T}^\flat
 +\Psi -{\rm Def}_{\mathcal D}\,H \eq 0\,,\\
\nonumber
  d_{\mathcal D}\,H
  +\widetilde{\Div}\,\tilde T -\sum\nolimits_a
 \big(\tilde A_{a}\tilde T^\sharp_{a} +\tilde T^\sharp_{a}\tilde A_{a}\big)^\flat \eq 0.
\end{eqnarray}
The equations for $r_{\,\widetilde{\mathcal D}}$ are dual to \eqref{E-genRicN}:
\begin{eqnarray}\label{E-genRicN-dual}
 r_{\,{\widetilde{\mathcal D}}} -\Div h -\<H, h\> +{\mathcal A}^\flat +{\mathcal T}^\flat
 +\widetilde\Psi -{\rm Def}_{\widetilde{\mathcal D}}\,\tilde H \eq 0\,, \\
\nonumber
 d_{\widetilde{\mathcal D}}\,\tilde H +\Div^\perp T -\sum\nolimits_{\,i}
 \big(A_{i} T^\sharp_{i} +T^\sharp_{i} A_{i}\big)^\flat \eq 0,
\end{eqnarray}
\end{prop}

\begin{proof}
Let $X,Y\in{\mathfrak X}_{\mathcal D}$ and $U,V\in{\mathfrak X}_{\widetilde{\mathcal D}}$,
then, see \cite[Lemma~2.25]{rov-m},
\begin{equation}\label{E-genricA}
 g(R^\nabla(U,\,X)V,\,Y) =
 g(((\nabla_U\,\tilde C)_V-\tilde C_V\tilde C_U)(X),Y) + g(((\nabla_X\,C)_Y-C_Y\,C_X)(U),V).
\end{equation}
Assume $\nabla_X\,Y\in\widetilde{\mathcal D}_x$ and $\nabla_X\,E_a\in{\mathcal D}_x$ at $x\in M$.
Note~that
\begin{eqnarray*}
 \sum\nolimits_a g((\nabla_X\,C)_Y(E_a),\,E_a) \eq\!\sum\nolimits_a\!\nabla_X (g(C_Y(E_a),\,E_a))\\
 =\nabla_X \big(g\,\big(\sum\nolimits_a h(E_a,E_a),\,Y\big)\big) \eq g(\nabla_X H,\,Y).
\end{eqnarray*}
Thus, tracing \eqref{E-genricA} over $\widetilde{\mathcal D}_x$ yields
\begin{equation}\label{E-genCRicN}
 {r}_{\,{\mathcal D}}(X,Y) = g(\widetilde{\Div}\,\tilde C(X),Y) -g(\sum\nolimits_{a}\tilde C_{a}^2(X),Y)
 +g(\nabla_X H, Y) -\tr_{g}(C_Y C_X)\,,
\end{equation}
where $\widetilde{\Div}\,\tilde C=\sum\nolimits_{a=1}^n (\nabla_a\,\tilde C)_{a}$.
Using \eqref{E-genCRicN}, equalities
$\tr_{g}(A_Y T^\sharp_X)=0=\tr_{g}(T^\sharp_Y A_X)$
and definitions above, we find \eqref{E-genRicN} as the symmetric and antisymmetric parts of \eqref{E-genCRicN}.
\end{proof}

\begin{rem}\rm
Tracing \eqref{E-genRicN}$_1$ over $\widetilde{\mathcal D}$ and applying the equalities
\begin{eqnarray*}
 \tr_{g}\Psi \eq \sum\nolimits_{\,i}\tr_{g}(A_i^2 + (T^\sharp_i)^2) = \|h\|^2 - \|T\|^2,\\
 \tr_{g}\widetilde{\mathcal A}^\flat\eq\|\tilde h\|^2,\quad \tr_{g}\widetilde{\mathcal T}^\flat = -\|\tilde T\|^2,\\
 \tr_{g}\,({\Div}\,\tilde h) \eq \Div\tilde H,\quad
 \tr_{g}\,({\rm Def}_{\mathcal D}\,H) = \Div^\perp H =\Div H +\|H\|^2
\end{eqnarray*}
yield the remarkable formula by P.\,Walczak \cite{wa1},
\begin{equation}\label{eq-ran}
 \Sm_{\,\rm mix} = \|H\|^2-\|h\,\|^2+\|T\|^2+\|\tilde H\|^2-\|\tilde h\,\|^2+\|\tilde T\|^2
 +\Div(H+\tilde H)\,,
\end{equation}
which represents $\Sm_{\,\rm mix}$ in terms of extrinsic geometry of the distributions.
The norms of tensors can be calculated using the adapted orthonormal basis as
\begin{eqnarray*}
 \|\tilde h\,\|^2 \eq\sum\nolimits_{\,i,j}\|\tilde h({\mathcal E}_i,{\mathcal E}_j)\|^2,\quad
 \|\tilde T\|^2=\sum\nolimits_{\,i,j}\|\tilde T({\mathcal E}_i,{\mathcal E}_j)\|^2,\\
 \|h\,\|^2 \eq\sum\nolimits_{\,a,b}\|h({E}_a,{E}_b)\|^2,\quad
 \|T\|^2=\sum\nolimits_{\,a,b} \|T({E}_a,{E}_b)\|^2.
\end{eqnarray*}
\end{rem}

The difference, called the \textit{extrinsic curvature} of $\widetilde{\mathcal D}$
(of the leaves of $\calf$ in integrable~case),
\[
 R^{\,\rm ex}(X,\,Y,\,Z,\,W)=
 g(h(\tilde X,\tilde Z),h(\tilde Y,\tilde W))-g(h(\tilde X,\, \tilde Y),h(\tilde Z,\tilde W))
\]
is useful in study of extrinsic geometry of foliations, see \cite{rov-m,rw-m}.
The extrinsic Ricci tensor $\Ric^{\,\rm ex}$ and the extrinsic scalar curvature function $\Sm_{\,\rm ex}$ of
$\widetilde{\mathcal D}$ (resp., a foliation $\calf$) are
\begin{eqnarray}\label{E-RR0}
\nonumber
 \Ric^{\rm ex}(X,\,Y)\eq
 \sum\nolimits_{a}\!\big(g(h(\tilde X,\tilde Y),h(E_a,E_a))-g(h(\tilde X,E_a),h(\tilde Y,E_a))\big),\\
 \Sm_{\,\rm ex} \eq \tr_{g} \Ric^{\rm ex} =\sum\nolimits_{a} \Ric^{\rm ex}(E_a,\,E_a)\,.
\end{eqnarray}
Tensors $\tilde R^{\,\rm ex}$, ${\widetilde\Ric}\,\!^{\,\rm ex}$ and $\widetilde\Sm_{\,\rm ex}$
are defined similarly for $\mathcal D$.
Using the equalities
\[
 \Sm_{\,\rm ex}=\|H\|^2-\|h\|^2,\quad \widetilde\Sm_{\,\rm ex} =\|\tilde H\|^2-\|\tilde h\|^2,
\]
we rewrite \eqref{eq-ran} as
\begin{equation}\label{eq-ran-ex}
 \Sm_{\,\rm mix} = \Sm_{\,\rm ex} +\widetilde\Sm_{\,\rm ex} +\|T\|^2 +\|\tilde T\|^2 + \Div(H+\tilde H)\,,
\end{equation}
Note that $\|H\|^2\le n\,\|h\,\|^2$ with the equality for totally umbilical $\widetilde{\mathcal D}$.

\subsection{Variation formulas for the extrinsic geometry}
\label{sec:prel}

To apply the methods of variational calculus to $J_{\,\rm mix}$,
given an adapted metric $g$ on $(M,\widetilde{\mathcal D},{\mathcal D})$, we shall consider smooth $1$-parameter variations of $g_0 = g$,
\begin{equation}\label{E-Sdtg}
 \big\{ g_t \in {\rm Riem}(M, \, \widetilde{\mathcal D}, \, {\mathcal D}) : |t| < \eps \big\}\,.
\end{equation}
If the integral \eqref{E-Jmix} does not converge then the induced infinitesimal variations
$S_t \equiv (\partial g_t/\partial t)\in {\mathfrak M}$ are supported in a sufficiently large relatively compact domain $\Omega\subset M$.
 We adopt the notations
\begin{equation}\label{E-Sdtg-2}
 \partial_t \equiv \partial/\partial t,\quad
 S_t \equiv \dt\,g_t,\quad
 S \equiv \{\dt g_t\}_{t=0}.
\end{equation}
By taking into account the decomposition (\ref{e:Decomposition2}), it will suffice to work with special curves
$\{ g_t \}_{|t| < \eps}$ issuing at $g \in {\rm Riem}(M, \, \widetilde{\mathcal D}, \, {\mathcal D})$ termed $\widetilde{\mathcal D}$-\textit{variations} (resp. $\mathcal D$-\textit{variations}),
as the associated infinitesimal variation $S$ lies in ${\mathfrak M}_{\mathcal D}$
(resp. in ${\mathfrak M}_{\widetilde{\mathcal D}}$).

An adapted variation of a metric $g\in{\rm Riem}(M,\,\widetilde{\mathcal D},\,{\mathcal D})$ has the form $\{ g_t^\bot + \tilde{g}_t : |t| < \eps \}$.
The~corresponding ${\mathcal D}$-variation (resp. $\widetilde{\mathcal D}$-variation) of $g$ has the form
\begin{equation}\label{e:Var}
 \big\{ g_t = g^\perp_t +\tilde g : \ |t| < \,\eps \big\},\quad
 ({\rm resp.}\ \big\{ g_t = g^\perp +\tilde g_t : \ |t| < \,\eps \big\}).
\end{equation}
For variations \eqref{E-Sdtg}--\,\eqref{E-Sdtg-2} we have, see for example~\cite{rw-m},
\begin{eqnarray}\label{eq2G}
 2\,g_t(\dt(\nabla^t_X\,Y), Z)\eq(\nabla^t_X\,S)(Y,Z)+(\nabla^t_Y\,S)(X,Z)-(\nabla^t_Z\,S)(X,Y),
\end{eqnarray}
where
$X,Y,Z\in\mathfrak{X}_M$
and the first covariant derivative of a $(0,2)$-tensor $S$ is expressed~as
\begin{eqnarray*}
 \nabla_{Z}\,S(Y,V)\eq Z(S(Y,V)) -S(\nabla_Z Y, V)-S(Y,\nabla_Z V).
\end{eqnarray*}

\begin{lem}\label{prop-Ei-a}
Let a local $(\widetilde{\mathcal D},\,{\mathcal D})$-adapted frame $\{E_a,\,{\mathcal E}_{i}\}$
evolve by \eqref{E-Sdtg}--\,\eqref{E-Sdtg-2} according to
 \begin{equation}\label{E-frameE}
 \dt E_a=-(1/2)\,S^\sharp(E_a),\qquad
 \dt {\mathcal E}_{i}=-(1/2)\,S^\sharp({\mathcal E}_{i}).
\end{equation}
 Then $\{E_a(t),{\mathcal E}_{i}(t)\}$ is a $g_t$-orthonormal frame adapted to
 $(\widetilde{\mathcal D},{\mathcal D})$ for all $\,t$.
\end{lem}

\proof
For $\{E_a(t)\}$ (and similarly for $\{{\mathcal E}_{i}(t)\}$) we have
\begin{eqnarray*}
 &&\dt(g_t(E_a, E_b)) = g_t(\dt E_a(t), E_b(t)) +g_t(E_a(t), \dt E_b(t))
 +(\dt g_t)(E_a(t), E_b(t))  \\
 &&= S(E_a(t), E_b(t))-\frac12\,g_t(S^\sharp(E_a(t)), E_b(t)) -\frac12\,g_t(E_a(t), S^\sharp(E_b(t)))=0.
 \qed
\end{eqnarray*}

\begin{lem}[Cf. \cite{rovwol}]\label{L-tildeHh}
For ${\mathcal D}$-variations \eqref{E-Sdtg}--\,\eqref{E-Sdtg-2} we have $\dt T = 0$, $\dt\tilde T = 0$ and
\begin{eqnarray}\label{eq-hatbH-1}
 2\,\dt\tilde h(X,Y) \eq (\tilde h-\tilde T)(S^\sharp(X),Y)  +(\tilde h+\tilde T)(X,S^\sharp(Y)) -\!\widetilde\nabla S(X,Y),\\
\label{eq-hatH}
 2\,\dt\tilde H \eq -\!\widetilde\nabla(\tr S^\sharp),\quad
 \dt h = -S^\sharp\circ h,\quad \dt H = -S^\sharp(H).
\end{eqnarray}
For ${\mathcal D}$-conformal variations, i.e., $\,S=s\,g^\perp$ where $s:M\to\RR$ is a smooth function, we~have
\begin{eqnarray}\label{eq-hatbH-1b}
 \dt\tilde h \eq s\,\tilde h -(1/2)\,(\widetilde\nabla s)\,g^\perp,\quad
 \dt\tilde H = -(p/2)\,\widetilde\nabla\,s,\\
\label{eq-bH-conf}
 \dt h \eq -s h,\qquad \dt H = -s H.
\end{eqnarray}
The formulas for $\widetilde{\mathcal D}$-variations
and $\widetilde{\mathcal D}$-conformal variations
are dual to \eqref{eq-hatbH-1}\,--\,\eqref{eq-hatH}
and \eqref{eq-hatbH-1b}\,--\,\eqref{eq-bH-conf}.
Hence, ${\mathcal D}$-variations preserve total umbilicity, total geodesy and harmonicity of
$\,\widetilde{\mathcal D}$.
Moreover, ${\mathcal D}$-confor\-mal variations preserve total umbilicity of $\,{\mathcal D}$.
\end{lem}

\begin{proof}
Recall that $\tilde C_N=\tilde A_N+\tilde T^\sharp_N$ is the conullity operator of ${\mathcal D}$ relative to the
unit vector $N\in\widetilde{\mathcal D}$.
Using (\ref{eq2G}), the symmetry of $S$, and the property $S(\cdot\,, \widetilde{\mathcal D})=0$, we obtain
\begin{eqnarray*}
 2\,g_t(\dt\tilde h(X,Y), N) \eq g_t(\dt(\nabla^t_X Y + \nabla^t_Y X), \,N)\\
 \eq (\nabla^t_X S)(Y,N)+(\nabla^t_Y S)(X,N) -(\nabla^t_{N} S)(X,Y)\\
 \eq S(\tilde C_N(X),Y)+S(\tilde C_N(Y),X)-(\nabla^t_{N} S)(X,Y)
\end{eqnarray*}
for all $X,Y\in{\mathcal D}$ and a unit vector $N\in\widetilde{\mathcal D}$. The above and
\[
 S(\tilde C_N(X),Y) = g((\tilde A_N+\tilde T^\sharp_N)(X), S^\sharp(Y))=g((\tilde h+\tilde T)(X, S^\sharp(Y)),\,N)
\]
yield \eqref{eq-hatbH-1}.
Next, tracing \eqref{eq-hatbH-1} and using skew-symmetry of $\tilde T$, we deduce \eqref{eq-hatH}$_1$,
\begin{eqnarray*}
 && 2\,g(\dt\tilde H,X) = 2\sum\nolimits_i g(\dt(\tilde h({\mathcal E}_i,{\mathcal E}_i)),X) \\
 && = 2\sum\nolimits_i g(\dt\tilde h({\mathcal E}_i,{\mathcal E}_i)
 +2\,\tilde h(\dt{\mathcal E}_i,{\mathcal E}_i),\,X)
 = -\sum\nolimits_i (\nabla_X\,S)({\mathcal E}_i,{\mathcal E}_i) =-X(\tr S^\sharp).
\end{eqnarray*}
By \eqref{eq2G}, for any $X\in {\mathcal D}$ and $E_a, \, E_b\in {\mathfrak X}_{\widetilde{\mathcal D}}$,
\begin{eqnarray*}
 2\,g_t(\dt h(E_a, E_b), X) \eq 2\,g_t(\dt(\nabla^t_{a} E_b +\nabla^t_{b} E_a),\ X)\\
 \eq(\nabla^t_{a}\,S)(X, E_b)+(\nabla^t_{b}\,S)(X, E_a) -(\nabla^t_X S)(E_a, E_b) \\
 \eq -S(\nabla^t_{a} E_b, X) -S(\nabla^t_{b} E_a, X) =-2\,S(h(E_a, E_b), X).
\end{eqnarray*}
This proves \eqref{eq-hatH}$_2$.
 Next (by $S(E_a,\,\cdot)=0$ and the equality $H=\sum_a h(E_a ,\,E_a)$ for a~local orthonormal frame $(E_a)$ of $\widetilde{\mathcal D}$) we may derive \eqref{eq-hatH}$_3$
\begin{eqnarray*}
 g_t (\dt H, X) \eq \sum\nolimits_a g_t(\dt(\nabla_{a} E_a), X)
  =(\nabla_{a}\,S)(E_a, X) -(1/2)\,(\nabla_X\,S)(E_a , E_a)\\
  \eq-\sum\nolimits_a S(\nabla_{a} E_a, X) = -S( H, X) =-g(S^\sharp(H),\,X).
\end{eqnarray*}
For ${\mathcal D}$-conformal variations we substitute $S=s\,g^\perp$ into \eqref{eq-hatbH-1} and obtain \eqref{eq-hatbH-1b}, also \eqref{eq-bH-conf} follows directly from \eqref{eq-hatH}.
Recall that total umbilicity of $\widetilde{\mathcal D}$ means $h =\frac1p\,H\,g_{|\,\widetilde{\mathcal D}}$.
By \eqref{eq-hatH} we have
\[
 \dt(h -(1/p)H\,g_{|\,\widetilde{\mathcal D}}) = -S^\sharp\circ (h -(1/p)H\,g_{|\,\widetilde{\mathcal D}}).
\]
Thus, the last claim in the lemma is a consequence of results about solutions to ODEs.
\end{proof}

Define ${\mathcal D}$-truncated symmetric $(0,2)$-tensor $\Phi_h$ using the identity (with arbitrary~$S$)
\begin{equation}\label{E-Phi}
 \<\Phi_h,\ S \>= S(H,\,H) -\sum\nolimits_{\,a,\,b} S(h(E_a,E_b), h(E_a,E_b))
\end{equation}
that vanishes when $n=1$.
We have
 $\tr_{g}\Phi_h=\Sm_{\,\rm ex}$ and $\tr_{g}\Phi_T=-\|T\|^2$.
Using the skew-symmetry of $T$, define the ${\mathcal D}$-truncated symmetric $(0,2)$-tensor $\Phi_T$ by
\begin{equation}\label{E-PhiT}
 \<\Phi_T,\ S \>=-\sum\nolimits_{\,a,\,b} S(T(E_a,E_b), T(E_a,E_b)).
\end{equation}
Similarly, define $\widetilde{\mathcal D}$-truncated tensors $\Phi_{\tilde h}$ and $\Phi_{\tilde T}$, and
get
 $\tr_{g}\Phi_{\tilde h}=\widetilde{\Sm}_{\,\rm ex}$ and
 $\tr_{g}\Phi_{\tilde T}=-\|\tilde T\|^2$.

Define a self-adjoint $(1,1)$-tensor with zero trace (and similarly its dual tensor ${\cal K}$)
\[
 \tilde{\cal K} =\sum\nolimits_{\,a}\,[\tilde T^\sharp_a , \tilde A_a] 
 =\sum\nolimits_{\,a}\,(\tilde T^\sharp_a \tilde A_a - \tilde A_a \tilde T^\sharp_a).
\]
Observe that if ${\cal D}$ is integrable then $\tilde T^\sharp_a = 0$ for all $a \in \{1, \ldots, n\}$, hence $\tilde{\cal K} =0$. Also, if ${\cal D}$ is totally umbilical, then every operator $\tilde A_a$ is a multiple of identity and $\tilde{\cal K}$ vanishes as well.

\begin{lem}\label{L-H2h2-D}
For ${\mathcal D}$-variations \eqref{E-Sdtg}--\,\eqref{E-Sdtg-2} we have
\begin{eqnarray}\label{E-h2T2-D1}
\nonumber
 &&\dt\,\|\,\tilde h\,\|^2 = \<\Div\tilde h -\tilde{\cal K}^\flat,\, S\> -\Div(\<\tilde h,\,S\>),\\
\nonumber
 &&\dt\,\|\tilde H\,\|^2 = \<(\Div\tilde H)\,g,\, S\> -\Div((\tr_{g} S)\tilde H),\\
\nonumber
 &&\dt\,(\|\,h\,\|^2-\|H\,\|^2) = \<\Phi_h,\ S\>,\\
 &&\dt\,\|\,\tilde T\,\|^2 = \<2\,\widetilde{\mathcal T}^\flat,\ S\>, \;\;
   \dt\,\|\,T\,\|^2 =-\<\Phi_T,\ S\>\,.
\end{eqnarray}
For ${\mathcal D}$-conformal variations \eqref{E-Sdtg}--\,\eqref{E-Sdtg-2}, we have
\begin{eqnarray}\label{E-h2T2-conf1}
\nonumber
 \dt\,\|\,\tilde h\,\|^2 \eq s\Div\tilde H -\Div(s\tilde H), \\
\nonumber
 \dt\,\|\tilde H\,\|^2 \eq p\,\big(s\Div\tilde H -\Div(s\tilde H)\big), \\
\nonumber
 \dt\,\|\,h\,\|^2 \eq -s\,\|h\,\|^2, \quad
 \dt\,\|H\,\|^2 = -s\,\|H\,\|^2,\\
 \dt\,\|\,\tilde T\,\|^2 \eq -2s\,\|\,\tilde T\,\|^2,\quad
 \dt\,\|\,T\,\|^2 = s\,\|T\|^2\,.
\end{eqnarray}
The formulas for $\widetilde{\mathcal D}$-variations are dual to \eqref{E-h2T2-D1} and \eqref{E-h2T2-conf1}.
\end{lem}

\begin{proof}
Assume $\nabla_{a}\,{\mathcal E}_i\in{\mathcal D}_x$ at a point $x\in M$.
We calculate using \eqref{eq2G} and Lemmas~\ref{prop-Ei-a} and~\ref{L-tildeHh}.
First we obtain~\eqref{E-h2T2-D1}$_1$:
\begin{eqnarray*}
 \dt\,\|\,\tilde T\,\|^2 \eq 2\sum\nolimits_{\,i,j,a}g(\tilde T({\mathcal E}_i,{\mathcal E}_j), E_a)
 \,g\big(\tilde T({\dt\mathcal E}_i,{\mathcal E}_j) +\tilde T({\mathcal E}_i,\dt{\mathcal E}_j), E_a\big)  \\
 \eq -\sum\nolimits_{\,i,j,a}g(\tilde T({\mathcal E}_i,{\mathcal E}_j), E_a)
 \,g\big(\tilde T(S^\sharp({\mathcal E}_i),{\mathcal E}_j)+\tilde T({\mathcal E}_i,S^\sharp({\mathcal E}_j)), E_a\big)  \\
 \eq -\!\sum\nolimits_{\,i,j,a}g(\tilde T^\sharp_a({\mathcal E}_i),{\mathcal E}_j)
 \,g((\tilde T^\sharp_a S^\sharp +S^\sharp\tilde T^\sharp_a)({\mathcal E}_i), {\mathcal E}_j) \\
 \eq -\!\sum\nolimits_{\,i,a}g(
 (\tilde T^\sharp_a S^\sharp{+}S^\sharp\tilde T^\sharp_a)({\mathcal E}_i), \tilde T^\sharp_a({\mathcal E}_i))
 =\sum\nolimits_{\,i,a}g(
 ((\tilde T^\sharp_a)^2S^\sharp+\tilde T^\sharp_a S^\sharp\tilde T^\sharp_a)({\mathcal E}_i),\,{\mathcal E}_i) \\
 \eq 2\sum\nolimits_{\,a}\tr_{g}((\tilde T^\sharp_a)^2 S^\sharp)
 =2\tr_{g}(\widetilde{\mathcal T} S^\sharp) =\<2\,\widetilde{\mathcal T}^\flat,\ S\>.
\end{eqnarray*}
Next, using \eqref{E-divP}, we obtain \eqref{E-h2T2-D1}$_1$:
\begin{eqnarray*}
 \dt\,\|\,\tilde h\,\|^2 \eq 2\sum\nolimits_{\,i,j,a}g(\tilde h({\mathcal E}_i,{\mathcal E}_j),\, E_a)
 g(\dt(\tilde h({\mathcal E}_i,{\mathcal E}_j)),\, E_a)  \\
 \eq 2\sum\nolimits_{\,i,j,a}g(\tilde h({\mathcal E}_i,{\mathcal E}_j),\, E_a)
 \,g\big((\dt\tilde h)({\mathcal E}_i,{\mathcal E}_j)
 +\tilde h(\dt{\mathcal E}_i,{\mathcal E}_j)
 +\tilde h({\mathcal E}_i,\dt{\mathcal E}_j),\, E_a\big)  \\
  \eq\!\sum\nolimits_{\,i,j,a}g(\tilde h({\mathcal E}_i,{\mathcal E}_j),\, E_a)
 \,\big(g(\tilde T({\mathcal E}_i,S^\sharp({\mathcal E}_j))-\tilde T(S^\sharp({\mathcal E}_i),{\mathcal E}_j),\,E_a) -\nabla_a\,S({\mathcal E}_i,{\mathcal E}_j)\big)  \\
  \eq \sum\nolimits_{\,i,j,a} \Big(
 g(\tilde A_a({\mathcal E}_i),\,{\mathcal E}_j)\,g([S^\sharp, \tilde T^\sharp_{a}]({\mathcal E}_i),\,{\mathcal E}_j)
 -\nabla_a\,\big(g(S({\mathcal E}_i,{\mathcal E}_j)\,\tilde h({\mathcal E}_i,{\mathcal E}_j),\, E_a)\big)   \\
 \minus \nabla_a\,g(\tilde h({\mathcal E}_i,{\mathcal E}_j),\, E_a)\,S({\mathcal E}_i,\,{\mathcal E}_j)\Big)
 = \<\widetilde{\Div}\,\tilde h - g(\tilde h,\tilde H) +\tilde{\cal K}^\flat,\, S\> -\Div(\<\tilde h,\,S\>).
\end{eqnarray*}
Next, using $S(\tilde H,\tilde H)=0$ (since $S$ is ${\mathcal D}$-truncated) we obtain
\begin{eqnarray*}
 \dt\,\|\,\tilde H\,\|^2 \eq \dt g(\tilde H,\,\tilde H) =2\,g(\dt\tilde H,\,\tilde H)
 =-g(\nabla(\tr S^\sharp),\,\tilde H).
\end{eqnarray*}
Note that $g(\nabla(\tr S^\sharp),\,\tilde H)=\Div((\tr S^\sharp)\tilde H)-(\Div\tilde H)\tr S^\sharp$;
hence, \eqref{E-h2T2-D1}$_2$ follows. We have
\begin{eqnarray*}
 \dt\,\|\,H\,\|^2 \eq \dt g(H,\,H) =S(H,H) + 2\,g(\dt H,\,H)  \\
  \eq S(H,H) -2\,g(S^\sharp(H),\,H) =-S(H,H), \\
\dt\,\|\,h\,\|^2 \eq \dt\sum\nolimits_{\,i,\,a,\,b} g(h(E_a,\,E_b),\,{\mathcal E}_i)^2 \\
 \eq 2\sum\nolimits_{\,i,\,a,\,b} g(h(E_a,\,E_b),\,{\mathcal E}_i)\,\dt g(h(E_a,\,E_b),\,{\mathcal E}_i)   \\
 \eq -\!\sum\nolimits_{\,i,\,a,\,b} g(h(E_a,\,E_b),\,{\mathcal E}_i)\,
 g(h(E_a,\,E_b),\,S^\sharp({\mathcal E}_i))  \\
 \eq -\!\sum\nolimits_{\,a,\,b} S(h(E_a,\,E_b),\,h(E_a,\,E_b))\,.
\end{eqnarray*}
From the above, \eqref{E-h2T2-D1}$_3$ follows. Finally, we have \eqref{E-h2T2-D1}$_4$:
\begin{eqnarray*}
 \dt\,\|\,T\,\|^2 \eq \dt\sum\nolimits_{\,i,\,a,\,b} g(T(E_a,\,E_b),\,{\mathcal E}_i)^2 \\
 \eq 2\sum\nolimits_{\,i,\,a,\,b} g(T(E_a,\,E_b),\,{\mathcal E}_i)\,\dt(g(T(E_a,\,E_b),\,{\mathcal E}_i))  \\
 \eq 2\sum\nolimits_{\,i,\,a,\,b} g(T(E_a,\,E_b),\,{\mathcal E}_i)\,
 \big( S(T(E_a,\,E_b),\,{\mathcal E}_i) +g(T(E_a,\,E_b),\,\dt{\mathcal E}_i)\big)   \\
 \eq \sum\nolimits_{\,i,\,a,\,b} g(T(E_a,\,E_b),\,{\mathcal E}_i)\,g(T(E_a,\,E_b),\,S^\sharp({\mathcal E}_i)) \\
 \eq \sum\nolimits_{\,a,\,b} S(T(E_a,\,E_b),\,T(E_a,\,E_b)).
\end{eqnarray*}
For ${\mathcal D}$-conformal case we use $\nabla_a\,S=(\nabla_a\,s)\,g^\perp$ and the above.
\end{proof}

\subsection{Directional derivatives of $J_{\,\rm mix}$ on almost-product manifolds}
\label{subsec:EU-general}

In this section we compute directional derivatives $D_g J_{\,\rm mix}$ of the functional \eqref{E-DJ}
on a closed Riemannian manifold $(M,g)$ with almost-product structure.
 For any $f \in L^1 ({M}, {\rm d}\, {\rm vol}_g )$, denote~by
\[
 f({M}, g) := \frac{1}{{\rm Vol}({M}, g)} \int_{M} f \, {\rm d} \, {\rm vol}_g
\]
the mean value of $f$ on ${M}$ with respect to ${\rm d} \, {\rm vol}_g$.
 Together with $g_t$ of \eqref{e:Var}, consider the metrics
\begin{eqnarray}\label{e:Var-Bar}
 \bar g_t \eq \phi_t g^\perp_t + \tilde g, \;\;\;
 \phi_t \equiv \big({\rm Vol}({M}, g_t)/{\rm Vol}({M}, g)\big)^{-2/p}\,,\ |t|<\eps\\
\nonumber
 {\rm respectively,}\ \ \bar g_t \eq g^\perp +\phi_t \tilde g_t,\quad
 \phi_t \equiv \big({\rm Vol}({M}, g_t)/{\rm Vol}({M}, g)\big)^{-2/n}\,,\ |t|<\eps\,.
\end{eqnarray}
We will show that ${\rm Vol}({M},\bar g_t) = {\rm Vol}({M}, g)$ for all $t$.

\begin{prop}\label{P-1-5}
The $\mathcal D$-variations $($respectively, $\widetilde{\mathcal D}$-variations$)$ of the functional \eqref{E-Jmix}, corresponding to $\bar g_t$ and $g_t$,
are related~by
\begin{eqnarray}\label{E-varJh-f}
 {\rm\frac{d}{dt}}\big\{J_{\,\rm mix} (\bar{g}_t)\big\}_{t=0}
 ={\rm\frac{d}{dt}}\big\{J_{\,\rm mix} ({g}_t)\big\}_{t=0}
 - \frac12\,\Sm^*_{\,\rm mix}(M,g) \int_{M} \big(\tr_{g} S \big)\,{\rm d}\vol_g \,,
\end{eqnarray}
where
\begin{equation}\label{E-S-star}
 \Sm^*_{\,\rm mix}(M,g) =\left\{\begin{array}{cc}
  \Sm_{\,\rm mix}(M,g) -\frac2p\,(\Sm_{\,\rm ex} +2\,\|\tilde T\|^2 -\|T\|^2)(M,g) &
  {\rm for}~{\mathcal D}{\rm-variations},\\
  \Sm_{\,\rm mix}(M,g) -\frac2n\,(\widetilde\Sm_{\,\rm ex} -\|\tilde T\|^2 +2\,\|T\|^2)(M,g) &
  {\rm for}~\widetilde{\mathcal D}{\rm-variations}\,.
  \end{array}\right.
\end{equation}
\end{prop}

\proof
 By \eqref{eq-ran-ex} and the Divergence Theorem, we have
\begin{equation}\label{eq62}
 J_{\,\rm mix}(g)=\int_{M} Q(g)\ {\rm d}\vol_g\,,
\end{equation}
where $Q(g):=\Sm_{\,\rm mix} -\Div(H+\tilde H)$ is represented using \eqref{eq-ran-ex} as
\begin{equation}\label{E-Q-def}
 Q(g) = \Sm_{\,\rm ex}(g) +\widetilde\Sm_{\,\rm ex}(g) +\|T\|_{g}^2 +\|\tilde T\|_{g}^2\,.
\end{equation}
The volume form evolves as (cf.~\cite{rw-m})
\begin{equation}\label{E-dtdvol}
 \partial_t\,\big({\rm d}\vol_{g_t}\!\big) = \frac12\,(\tr_{g_t} S_t)\,{\rm d}\vol_{g_t}.
\end{equation}
Thus,
\begin{equation}\label{E-varJh-init}
 {\rm\frac{d}{dt}}\big\{J_{\,\rm mix} ({g}_t)\big\}_{t=0} =\int_{M} \!\Big\{\dt Q(g_t)_{|\,t=0}
 +\frac12\big(\Sm_{\,\rm mix}(g) -\Div(H + \tilde H)\big)\tr_{g} S\Big\}\,{\rm d}\vol_g\,.
\end{equation}
Let us fix a ${\mathcal D}$-variation \eqref{e:Var}$_1$,
the case of $\widetilde{\mathcal D}$-variations is studied similarly.
As $\bar g_t$ are ${\mathcal D}$-conformal to $g_t$  with constant scale $\phi_t$,
their volume forms are related~as
\begin{equation}\label{e:Conf-Vol}
 {\rm d}\vol_{\,\bar g_t} = \phi^{p/2}_t{\rm d}\vol_{\,g_t};
\end{equation}
hence, ${\rm Vol}({M},\bar g_t) = \int_{\,M} {\rm d}\vol_{\,\bar g_t} = {\rm Vol}({M}, g)$.
 Let us differentiate in (\ref{e:Conf-Vol}) so that to obtain
\begin{eqnarray*}
 \dt\;({\rm d}\vol_{\,\bar g_t}) \eq (\phi_t^{p/2})^\prime \; {\rm d}\vol_{g_t} + \phi_t^{p/2}\;
 \dt\,({\rm d}\vol_{g_t}) \\
 \eq \frac12\Big(\tr S^\sharp_t
 -\frac{1}{{\rm Vol}({M}, g_t)}\int_{M}(\tr_{g_t} S_t)\,{\rm d}\vol_{\,g_t}\Big)\,{\rm d}\vol_{\,\bar g_t}\,.
\end{eqnarray*}
Here we used (\ref{E-dtdvol}) and the fact that $\phi_0=1$ and
\begin{eqnarray*}
 \phi^\prime_t \eq -\frac2p\,\Big(\frac{{\rm Vol}({M},g_t)}{{\rm Vol}({M}, g)}\Big)^{-\frac2p-1}
 \frac{1}{{\rm Vol}({M}, g)} \int_{M} \dt({\rm d}\vol_{\,g_t})  \\
 \eq -\frac{\phi_t}{p\,{\rm Vol}({M}, g_t)}\int_{M} (\tr_{g_t} S_t)\,{\rm d}\vol_{\,g_t}.
\end{eqnarray*}
For the ${\mathcal D}$-scaling $\,\bar g = \phi\,g^\perp + \tilde g$ of $\,g = g^\perp + \tilde g$,
using \eqref{E-h2T2-conf1}, we have
\begin{eqnarray*}
 &&\|T\|_{\bar g}^2=\phi\,\|T\|_{g}^2,\quad
 \|\tilde T\|_{\bar g}^2=\phi^{-2}\|\tilde T\|_{g}^2,\\
 &&\|h\|_{\bar g}^2=\phi^{-1}\|h\|_{g}^2,\quad
 \|H\|_{\bar g}^2=\phi^{-1}\|H\|_{g}^2,\quad
 \|\tilde h\|_{\bar g}^2=\|\tilde h\|_{g}^2,\quad
 \|\tilde H\|_{\bar g}^2=\|\tilde H\|_{g}^2.
\end{eqnarray*}
Substituting into $Q(\bar g)$ due to definition \eqref{E-Q-def}, we obtain
\[
 Q(\bar g) = Q(g) +(\phi^{-1}-1)\,\Sm_{\,\rm ex}(g) +(\phi^{-2}-1)\,\|\tilde T\|_{g}^2 +(\phi-1)\,\|T\|_{g}^2\,.
\]
(For example, $Q(\bar g) = Q(g)$ when $n=1$ and $\tilde T=0$).
The derivation of the above yields
\[
 \dt Q(\bar g_t)_{|\,t=0} = \dt Q(g_t)_{|\,t=0}
 -\phi^{\,\prime}_0\,\big(\Sm_{\,\rm ex}(g) +2\,\|\tilde T\|_{g}^2 -\|T\|_{g}^2\big)\,,
\]
where $\phi^{\,\prime}_0=-\frac{1}p\,\frac{1}{{\rm Vol}({M}, g)}\int_{M}\,(\tr_{g} S)\,{\rm d}\vol_{\,g}$.
Hence, the $\mathcal D$-variation of the functional \eqref{E-Jmix}~is
\begin{eqnarray*}
 &&\hskip-10mm{\rm\frac{d}{dt}}\big\{J_{\,\rm mix} (\bar{g}_t)\big\}_{t=0} \\
 && = \int_{M} \Big\{\,\dt Q(\bar g_t)_{\,|\,t=0} +\frac12\,Q(g)\big(\tr_{g} S
 -\frac{1}{{\rm Vol}({M}, g)}\!\int_{M} (\tr_{g} S)\,{\rm d}\vol_g\big)\Big\}\,{\rm d}\vol_g  \\
 && = \!\int_{M}\dt Q(g_t)_{\,|\,t=0}\,{\rm d}\vol_g +\frac12\int_{M} Q(g)(\,\tr_{g} S)\,{\rm d}\vol_g
 +\int_{M}\big(\tr_{g} S\big)\,{\rm d}\vol_g \times\\
 &&
 \times\frac{1}{2\,{\rm Vol}({M}, g)} \int_{M}\Big[\,
 \frac2p\,\big(\Sm_{\,\rm ex} +2\,\|\tilde T\|_{g}^2 -\|T\|_{g}^2 \big)-Q(g)
  \Big]\,{\rm d}\vol_g \\
 && ={\rm\frac{d}{dt}}\big\{J_{\,\rm mix} ({g}_t)\big\}_{t=0}
 -\frac12\,\Sm^*_{\,\rm mix}(M,g)\int_{M} \big(\tr_{g} S \big)\,{\rm d}\vol_g \,,\qed
\end{eqnarray*}

\begin{rem}\rm
It should be observed that we work with two kinds of variations, (\ref{e:Var}) and (\ref{e:Var-Bar});
the last one of which preserves the volume of ${M}$.
Formulas containing $\Sm^*_{\,\rm mix}(M,g)$ correspond to $1$-parameter variations of the form (\ref{e:Var-Bar}).
To obtain similar formulas, corresponding to $1$-parameter variations of the form (\ref{e:Var}), one should merely delete the mean value terms $\Sm^*_{\rm mix}(M,g)$ in the previous~identities.
\end{rem}

\begin{example}[Variations not preserving the volume of $(M,g)$]\label{sec:homothehy}\rm
In general, variations $g_t$ of \eqref{e:Var} do not preserve the volume of $(M,g)$,
and for critical metrics one may obtain only trivial solutions $J_{\,\rm mix}(g)=0$.
 Let $\widetilde{\mathcal D}$ be a codimension one integrable distribution
(i.e., $p=1$ and $T=0$) on a~closed Riemannian manifold $(M^{n+1},g)$ with $g = g^\perp + \tilde g$,
see Sect.~\ref{subsec:codim1fol}.
Define the functions $\tau_1=\tr_{g} h$ and $\tau_2=\|h\|_g^2$.
For a canonical ${\mathcal D}$-variation $g_t= g^\perp_t + \tilde g$,
i.e., $g^\perp_t=(1+t)\,g^\perp_0\ (|t|<1)$,
we have $\dt g_t =s\,g^\perp_t$ with $s=1/(1+t)$, and ${\rm d}\vol_t=(1+t)^{1/2}{\rm d}\vol$.
Since $(\tau^2_1-\tau_2)_t=(\tau^2_1-\tau_2)_0/(1+t)$, we~find
\[
 J_{\,\rm mix}(g_t) =\int_{M}(\tau_1^{2}-\tau_2)\,{\rm d}\vol_t =(1+t)^{-\frac12}J_{\,\rm mix}(g)
 \ \Rightarrow \
 {\rm\frac{d}{dt}}\big\{J_{\,\rm mix} ({g}_t)\big\}_{t=0} =-\frac12\,J_{\,\rm mix}(g).
\]
Thus, if $g$ is a~critical point of the functional $J_{\,\rm mix}$, then $J_{\,\rm mix}(g)=0$.
\end{example}

\subsection{Euler-Lagrange equations for $J_{\,\rm mix}$ on almost-product manifolds}
\label{subsec:EU-critical}

In this section we  write down Euler-Lagrange equations of the variational principle $\delta J_{\,\rm mix}(g)=0$
on a closed manifold $M$ with almost-product structure.

\begin{prop}\label{T-main00}
Let $\widetilde{\mathcal D}$ and ${\mathcal D}$ be complementary distributions on a closed manifold~$M$.
If $g\in{\rm Riem}(M,\,\widetilde{\mathcal D},\,{\mathcal D})$ is a critical point of $J_{\rm mix}$
with respect to variations \eqref{E-Sdtg}--\,\eqref{E-Sdtg-2}, then the following Euler-Lagrange equations
$($relating the quantities of extrinsic geo\-metry$)$ are satisfied:
\begin{eqnarray}\label{E-main-0ih}
 &&\hskip-9mm
 \Div(\tilde h -\tilde H\,g^\perp)
 -2\,\widetilde{\mathcal T}^\flat +\Phi_h +\Phi_T +\tilde{\cal K}^\flat\\
\nonumber
 &&\hskip-6mm =\frac12\,\big( \Sm_{\,\rm ex}+\widetilde\Sm_{\,\rm ex}
 +\|T\|^2+\|\tilde T\|^2 -\Sm^*_{\,\rm mix}(M,g) \big)\,g^\perp
 \quad ({\rm for}~{\mathcal D}{\rm-variations}),\\
 \label{E-main-0i-dualh}
 &&\hskip-9mm
 \Div(h -H\,\tilde g)
 -2\,{\mathcal T}^\flat +\Phi_{\tilde h} +\Phi_{\tilde T} +{\cal K}^\flat\\
\nonumber
 &&\hskip-6mm =\frac12\,\big(
 \Sm_{\,\rm ex}+\widetilde\Sm_{\,\rm ex} +\|T\|^2+\|\tilde T\|^2-\Sm^*_{\,\rm mix}(M,g) \big)\,\tilde g
 \quad ({\rm for}~\widetilde{\mathcal D}{\rm-variations})\,.
\end{eqnarray}
Using the partial Ricci tensor, we rewrite \eqref{E-main-0ih} and \eqref{E-main-0i-dualh}
 as
\begin{eqnarray}\label{E-main-0i}
 &&\hskip-7mm {r}_{\,{\mathcal D}} -\<\tilde h,\tilde H\> +\widetilde{\mathcal A}^\flat -\widetilde{\mathcal T}^\flat
 +\Phi_h +\Phi_T +\Psi -{\rm Def}_{\mathcal D}\,H +\tilde{\cal K}^\flat \\
\nonumber
 &&=\frac12\,\big(  \Sm_{\,\rm mix} - \Sm^*_{\,\rm mix}(M,g) + \Div(\tilde H - H)\big)\,g^\perp
 \quad ({\rm for}~{\mathcal D}{\rm-variations}),\\
 \label{E-main-0i-dual}
 &&\hskip-7mm {r}_{\,\widetilde{\mathcal D}} -\<h, H\> +{\mathcal A}^\flat -{\mathcal T}^\flat
 +\Phi_{\tilde h} +\Phi_{\tilde T} +\widetilde\Psi -{\rm Def}_{\widetilde{\mathcal D}}\,\tilde H +{\cal K}^\flat \\
\nonumber
 &&=\frac12\,\big( \Sm_{\,\rm mix} - \Sm^*_{\,\rm mix}(M,g) +\Div(H -\tilde H)\big)\,\tilde g
  \quad ({\rm for}~\widetilde{\mathcal D}{\rm-variations}).
\end{eqnarray}
\end{prop}

\begin{proof}
Since \eqref{E-main-0i-dualh} is dual to \eqref{E-main-0ih}, we shall prove \eqref{E-main-0ih} only.
By Lemma~\ref{L-H2h2-D} to \eqref{E-Q-def}, and using \eqref{E-divP} and the Divergence Theorem, we obtain
\begin{equation}\label{E-Sc-var1a}
 \int_{M}\dt Q\ {\rm d}\vol_g = \int_{M}\big\<
 -\Div(\tilde h -\tilde H\,g^\perp) +2\,\widetilde{\mathcal T}^\flat -\Phi_h -\Phi_T
 -\tilde{\cal K}^\flat,\ S\big>\,{\rm d}\vol_g,
\end{equation}
where $S=\{\dt g_t\}_{\,|\,t=0}\in{\mathfrak M}$.
Note that $\tr_{g} S=\<S,\,g\>$. Then by \eqref{E-varJh-init}, we have
\begin{eqnarray}\label{E-varJh-init2}
\nonumber
 {\rm\frac{d}{dt}}\big\{J_{\,\rm mix} ({g}_t)\big\}_{t=0} \eq\int_{M}
 \big\< -\Div(\tilde h -\tilde H\,g^\perp) +2\,\widetilde{\mathcal T}^\flat -\Phi_h -\Phi_T -\tilde{\cal K}^\flat\\
 \plus\frac12\big(\Sm_{\,\rm mix} -\Div(H + \tilde H)\big)\,g,\ S\big>\,{\rm d}\vol_g\,.
\end{eqnarray}
By \eqref{E-varJh-init2} and Proposition~\ref{P-1-5}, we obtain
\begin{eqnarray}\label{E-Jmix-dt-fin}
\nonumber
 {\rm\frac{d}{dt}}\,\big\{J_{\,\rm mix} (\bar g_t)\big\}_{t=0} \eq \!\int_{M}\Big\<
 -\Div(\tilde h -\tilde H\,g^\perp) +2\,\widetilde{\mathcal T}^\flat-\Phi_h-\Phi_T -\tilde{\cal K}^\flat\\
 \plus\frac12\,\big(\Sm_{\rm mix} -\Sm^*_{\,\rm mix}(M,g) -\Div(\tilde H +H) \big)\,g,\ S\Big\>\,{\rm d}\vol_g.
\end{eqnarray}
If $g$ is a critical point of $J_{\,\rm mix}$ with respect to $\mathcal D$-variations,
then the integral in \eqref{E-Jmix-dt-fin} is zero for arbitrary $\mathcal D$-truncated tensor $S\in{\mathfrak M}$,
that yields
\begin{equation}\label{E-main-0ih-temp}
 \Div(\tilde h -\tilde H\,g^\perp) -2\,\widetilde{\mathcal T}^\flat +\Phi_h +\Phi_T +\tilde{\cal K}^\flat
 =\frac12\,\big(\Sm_{\,\rm mix} -\Sm^*_{\,\rm mix}(M,g) - \Div(\tilde H + H) \big)\,g^\perp.
\end{equation}
The dual equation (i.e., for $\widetilde{\mathcal D}$-variations) is
\begin{equation}\label{E-main-0i-dualh-temp}
 \Div(h -H\,\tilde g) -2\,{\mathcal T}^\flat +\Phi_{\tilde h} +\Phi_{\tilde T} +{\cal K}^\flat
 =\frac12\,\big( \Sm_{\,\rm mix} -\Sm^*_{\,\rm mix}(M,g) - \Div\,(H +\tilde H) \big)\,\tilde g.
\end{equation}
Replacing  $\Sm_{\,\rm mix} - \Div\,(H +\tilde H)$ due to \eqref{eq-ran-ex}, we obtain \eqref{E-main-0ih}\,--\,\eqref{E-main-0i-dualh}. Replacing ${\Div}\,\tilde h$ and $\Div h$ in \eqref{E-main-0ih-temp}\,--\,\eqref{E-main-0i-dualh-temp} due to \eqref{E-genRicN}$_1$ and its dual \eqref{E-genRicN-dual}$_1$, we obtain \eqref{E-main-0i}\,--\,\eqref{E-main-0i-dual}.
\end{proof}

\begin{rem}\label{R-int}\rm
Tracing \eqref{E-main-0ih} (and similarly for \eqref{E-main-0i-dualh}), we obtain the equality
 \[
 \mathfrak{R} - \mathfrak{R}(M,g) = 2(1-p)\Div\tilde H +p\,\Div H,
 \]
 where $\mathfrak{R}= p\,\Sm_{\,\rm mix} +4\,\|T\|^2 -2\,\|\tilde T\|^2 +2\,\Sm_{\,\rm ex}$.
 Integrating this over a closed $M$ gives identity.
\end{rem}

\begin{rem}\label{Ex-12}\rm
If $g\in{\rm Riem}(M,\,\widetilde{\mathcal D},\,{\mathcal D})$ is a critical point
of the functional \eqref{E-Jmix} with respect to biconformal variations, then
\begin{eqnarray}\label{E-main-0conf}
 &&(p-2)\,\Sm_{\,\rm ex} +p\,\widetilde\Sm_{\,\rm ex} +(p+2)\,\|T\|^2 +(p-4)\,\|\tilde T\|^2
 +2(p-1)\,\Div\tilde H \\
\nonumber
 &&\hskip32mm  = p\,\Sm^*_{\,\rm mix}(M,g)
 \quad ({\rm for}~{\mathcal D}{\rm-conformal\ variations}),\\
 \label{E-main-0conf-dual}
 && (n-2)\,\widetilde\Sm_{\,\rm ex} +n\,\Sm_{\,\rm ex} +(n+2)\,\|\tilde T\|^2 +(n-4)\,\|T\|^2
 +2(n-1)\,\Div H \\
\nonumber
 &&\hskip32mm  = n\,\Sm^*_{\,\rm mix}(M,g)\quad ({\rm for}~\widetilde{\mathcal D}{\rm-conformal\ variations}).
\end{eqnarray}
Since \eqref{E-main-0conf-dual} is dual to \eqref{E-main-0conf}, we shall show \eqref{E-main-0conf} only.
For $S=s\,g^\perp$ with $s\in C^\infty(M)$ we have $\tr_{g} S=ps$.
Applying \eqref{E-h2T2-conf1} of Lemma~\ref{L-H2h2-D} to \eqref{eq-ran-ex}, similarly to \eqref{E-Sc-var1a},
we obtain
\begin{equation}\label{E-Sc-var1b}
 \int_{M}\dt Q(g_t)_{\,|\,t=0}\,{\rm d}\vol_g = \int_{M} s\,\big( (p-1)\Div\tilde H -\Sm_{\,\rm ex}
 -2\,\|\tilde T\|^2_{g} +\|T\|^2_{g} \big)\,{\rm d}\vol_g.
\end{equation}
By Proposition~\ref{P-1-5} and \eqref{eq-ran-ex}, we obtain
\begin{eqnarray}\label{E-Jmix-dt-fin2}
\nonumber
 {\rm\frac{d}{dt}}\big\{J_{\,\rm mix} (\bar{g}_t)\big\}_{t=0}
 \eq \int_{M} s\Big( (p-1)\Div\tilde H -\Sm_{\,\rm ex} -2\,\|\tilde T\|^2_{g} +\|T\|^2_{g} \\
 \plus \frac p2\,\big( \Sm_{\,\rm ex} +\widetilde\Sm_{\,\rm ex} +\|T\|^2_{g} + \|\tilde T\|^2_{g}
 -\Sm^*_{\,\rm mix}(M,g) \big)\Big) \,{\rm d}\vol_g.
\end{eqnarray}
If $g$ is a critical point of the functional $J_{\,\rm mix}$ with respect to biconformal variations, then
the integral in \eqref{E-Jmix-dt-fin2} is zero for arbitrary $s\in C^\infty(M)$, and \eqref{E-main-0conf} holds.
\end{rem}

\section{Applications to foliated manifolds}
\label{sec:main}

In this section we assume that $(M, g)$ is a semi-Riemannian manifold endowed with a $n$-dimensi\-onal foliation $\calf$ (i.e., $T=0$) such that $T{\mathcal F}_x$ is nondegenerate in $(T_x M , \, g_x)$ for every $x \in M$.
Then the normal distribution ${\mathcal D}$ is also nondegenerate
in $(T_x M,\, g_x)$. We have $\Phi_T=0$ and $\Psi(X,Y)=\tr_{g}(A_Y A_X)$.
Since $\widetilde{\mathcal D}=T\calf$, we write ${r}_{\,\calf}={r}_{\,\widetilde{\mathcal D}}$, see \eqref{E-Rictop},
and equations \eqref{E-genRicN-dual} read
\begin{eqnarray}\label{E-genRicN-2}
 {r}_{\,\calf} = \Div h +\<H, h\> -{\mathcal A}^\flat -\widetilde\Psi +{\rm Def}_{\calf}\,\tilde H\,,\quad
 d_\calf\,\tilde H=0\,.
\end{eqnarray}

\begin{example}\rm
Let $\calf$ be a one-dimensional foliation (i.e., $n=1$) on $(M^{p+1},g)$
spanned by a unit vector field $N$, then $\Sm_{\,\rm mix}=\Ric_g(N,N)$, and \eqref{E-Jmix} reduces to~\eqref{E-Jmix-N}. We have
\[
 r_{\calf}=\Ric_g(N,N)\,\tilde g,\qquad
 r_{\,\mathcal D}=(R_N)^\flat,
\]
where $R_N\!=R^\nabla(N,\,\cdot\,)N$ is the \textit{Jacobi operator}.
Let $\tilde h$ be the scalar second fundamental form of~${\mathcal D}$.
Define the functions $\tilde\tau_i=\tr(\tilde A_N^{\,i})\ (i\ge0)$.
 It is easy to see that $\widetilde{\Sm}_{\,\rm ex}=\tilde\tau_1^2 -\tilde\tau_2$ and
\begin{eqnarray*}
 \Div N \eq \sum\nolimits_{\,i} g(\nabla_i N, {\mathcal E}_i)=-g(N,\sum\nolimits_{\,i}\!\nabla_i\,{\mathcal E}_i)=-\tilde\tau_1,\\
 \Div(\tilde\tau_1 N) \eq N(\tilde\tau_1) +\tilde\tau_1\Div N=N(\tilde\tau_1)-\tilde\tau_1^2.
\end{eqnarray*}
 Notice that $(H^\flat\otimes H^\flat)(X,Y)=g(H,X)\,g(H,Y)$.
 One may calculate
\begin{eqnarray*}
 \widetilde{\mathcal A} \eq \tilde A_N^2,\quad
  g(\tilde h N,\tilde H) = \tilde\tau_1\tilde h,\quad \Psi = H^\flat\otimes H^\flat,\quad
  \widetilde\Psi = (\tilde\tau_2 - \|\tilde T\|^2)\,\tilde g,\\
 {\mathcal A}^\flat \eq \|H\|^2\tilde g,\quad
 {\mathcal T}  = 0,\quad
 g(h,H) = \|H\|^2\tilde g ,\\
 H\eq\nabla_N\,N,\quad h=H\,\tilde g,\quad \|h\|=\|H\|, \\
 \tilde H\eq\tilde\tau_1 N,\quad \tilde\tau_1=\tr_{g}\tilde h,\quad
 \|\tilde h\|^2=\tilde\tau_2,\quad
  {\rm Def}_{\calf}\,\tilde H = N(\tilde\tau_1)\,\tilde g\,.
\end{eqnarray*}
We find $\Div(\tilde h\,N) =\nabla_N\,\tilde h -\tilde \tau_1\tilde h$ and $\Div h =(\Div H)\,\tilde g$.
 Then \eqref{E-genRicN}$_1$ and \eqref{eq-ran-ex} read
\begin{eqnarray}
\nonumber
 \big(R_N + \tilde A_N^2+(\tilde T^\sharp_N)^2\big)^\flat \eq \nabla_N\,\tilde h -H^\flat\otimes H^\flat
 +{\rm Def}_{\mathcal D}\,H,\\
\label{E-RicNs1aa}
 \Ric_g(N,N)\eq \Div(\tilde\tau_1 N + H) +\tilde\tau_1^2 -\tilde\tau_2 +\|\tilde T\|^2.
\end{eqnarray}
 If $M$ is a closed manifold, then by \eqref{E-RicNs1aa} and the Divergence Theorem, we represent \eqref{E-Jmix-N}~as
\begin{equation}\label{E-Jmix3-flow}
 J_{\,\rm mix}(g) =\int_{\,M}(\tilde\tau_1^2 -\tilde\tau_2 +\|\tilde T\|^2)\,{\rm d}\vol_g\,.
\end{equation}
Note that, by \eqref{E-divP}, $\Div(F\,N) = \nabla_N\,F - \tilde\tau_1 F$ for $(0,2)$-tensors $F$ defined on $M$.
\end{example}

\subsection{Euler-Lagrange equations for $J_{\,\rm mix}$ and critical adapted metrics}
\label{subsec:EL-fol}

For a foliation $\calf$ with $T\calf=\widetilde{\mathcal D}$, definitions \eqref{E-S-star} have the view
\begin{equation}\label{E-S*main}
 \Sm^*_{\,\rm mix}(M,g) =\left\{\begin{array}{cc}
  \Sm_{\,\rm mix}(M,g) -\frac2p\,\big(\Sm_{\,\rm ex} +2\,\|\tilde T\|^2\big)(M,g) &
  {\rm for}~{\mathcal D}{\rm-variations},\\
  \Sm_{\,\rm mix}(M,g)
 -\frac2n\,\big(\widetilde\Sm_{\,\rm ex} -\|\tilde T\|^2\big)(M,g) &
  {\rm for}~{T\calf}{\rm-variations}\,.
  \end{array}\right.
\end{equation}

Using Proposition~\ref{T-main00}, we obtain the following.

\begin{thm}\label{T-main01}
Let $\calf$ be a foliation with a transversal distribution ${\mathcal D}$  on a closed manifold~$M$.
If $g\in{\rm Riem}(M,\,{T\calf},\,{\mathcal D})$ is a critical point of $J_{\rm mix}$
with respect~to variations~\eqref{E-Sdtg}--\,\eqref{E-Sdtg-2}, then the following Euler-Lagrange equations
are satisfied:
\begin{eqnarray}\label{E-main-i-hF}
 \Div(\tilde h {-}\tilde H\,g^\perp) \!-2\widetilde{\mathcal T}^\flat \!+\Phi_h \!+\tilde{\cal K}^\flat
  = \frac12\big(\Sm_{\,\rm ex} \!+\widetilde\Sm_{\,\rm ex} \!+\|\tilde T\|^2 \!-\Sm^*_{\,\rm mix}(M,g)\big)\,g^\perp
 \ ({\rm for}~{\mathcal D}{\rm-variations}), \\
 \label{E-main-ii-hF}
 \Div(h -H\,\tilde g) +\Phi_{\tilde h} +\Phi_{\tilde T}
  = \frac12\big(\Sm_{\,\rm ex} +\widetilde\Sm_{\,\rm ex}+\|\tilde T\|^2 -\Sm^*_{\,\rm mix}(M,g)\big)\,\tilde g
 \quad({\rm for}~{T\calf}{\rm-variations}).
\end{eqnarray}
Using the partial Ricci tensor, we rewrite the Euler-Lagrange equations as
\begin{eqnarray}\label{E-main-i}
\nonumber
 &&\hskip-7mm {r}_{\,{\mathcal D}} -\<\tilde h,\tilde H\> +\widetilde{\mathcal A}^\flat
 -\widetilde{\mathcal T}^\flat +\Phi_h +\Psi -{\rm Def}_{\mathcal D}\,H +\tilde{\cal K}^\flat\\
 && = \frac12\,\big(\Sm_{\,\rm mix} -\Sm^*_{\,\rm mix}(M,g) +\Div(\tilde H - H) \big)\,g^\perp
 \quad({\rm for}~{\mathcal D}{\rm-variations}), \\
 \label{E-main-ii}
 \nonumber
 &&\hskip-7mm {r}_{\,\calf} -\<h,H\> +{\mathcal A}^\flat
 +\Phi_{\tilde h} +\Phi_{\tilde T} +\widetilde\Psi -{\rm Def}_{\calf}\,\tilde H \\
 && =\frac12\,\big( \Sm_{\,\rm mix} -\Sm^*_{\,\rm mix}(M,g) +\Div(H-\tilde H) \big)\,\tilde g
 \quad({\rm for}~{T\calf}{\rm-variations}).
\end{eqnarray}
\end{thm}

Certainly, these mixed field equations admit amount of solutions
(e.g., twisted products, see Sect.~\ref{subsec:dtwisted})
which might serve as models in theoretical physics, see discussion in \cite{bdrs}.

\begin{cor}\label{C-main02}
Let $\calf$ be a foliation spanned by a nonzero vector field $N$
with a transversal distribution ${\mathcal D}$ on a closed manifold~$M$.
If $g\in{\rm Riem}(M,\,{T\calf},\,{\mathcal D})$ is a~critical point of the functional \eqref{E-Jmix-N}
with respect to variations \eqref{E-Sdtg}, then the following  Euler-Lagrange equations~hold:
\begin{eqnarray}\label{E-main-1h}
\nonumber
 &&\hskip-11mm
 \Div((\tilde h -\tilde\tau_1\,g^\perp)N)
 -2\big(\tilde T^{\sharp\,2}_N\big)^\flat
 +[\tilde T_N^\sharp,\tilde A_N]^\flat \\
 && \ \, =\frac12\,(\tilde\tau_1^2 \!-\tilde\tau_2 +\|\tilde T\|^2 \!-\Sm^*_{\,\rm mix}(M,g))\,g^\perp
  \ ({\rm for}~{\mathcal D}{\rm-variations}),\\
\label{E-main-2h}
 &&\hskip-11mm \tilde\tau_1^2 - \tilde\tau_2 = 3\,\|\tilde T\|^2 -\Sm^*_{\,\rm mix}(M,g)
  \quad ({\rm for}~{T\calf}{\rm-variations}).
\end{eqnarray}
One may rewrite the equations using $R_N$ and $\Ric_g(N,N)$, as
\begin{eqnarray}\label{E-main-1i}
\nonumber
 &&\big( R_N +\tilde A_N^2 -\tilde T^{\sharp\,2}_N \big)^\flat \!-\tilde\tau_1\tilde h
 +H^\flat\otimes H^\flat -{\rm Def}_{\mathcal D}\,H  \\
 &&\hskip4mm =\frac12\,\big(\Ric_g(N,N) -\Sm^*_{\,\rm mix}(M,g)
 +\Div(\tilde\tau_1 N - H)\big)\,g^\perp
 \ \ ({\rm for}~{\mathcal D}{\rm-variations}),\\
\label{E-main-2i}
 &&\Ric_g(N,N) = -\Sm^*_{\,\rm mix}(M,g) +4\,\|\tilde T\|^2 +\Div(\tilde\tau_1 N + H)
 \ \ ({\rm for}~{T\calf}{\rm-variations}).
\end{eqnarray}
\end{cor}

\begin{proof}
Substituting the values
 $\Phi_h=0=\Sm_{\,\rm ex}$, $\widetilde{\Sm}_{\,\rm ex}=\tilde\tau_1^2 -\tilde\tau_2$ and
 $\widetilde{\mathcal T} = \tilde T^{\sharp\,2}_N$
into \eqref{E-main-i-hF} yields \eqref{E-main-1h}.
Substituting the values $h =H\,\tilde g$,
 $\Phi_{\tilde h} = (\tilde\tau_1^2-\tilde\tau_2)\,\tilde g$
 and $\Phi_{\tilde T}  = -\|\tilde T\|^2 \tilde g$
into \eqref{E-main-ii-hF} yields \eqref{E-main-2h}.
Then, replacing terms in \eqref{E-main-1h} and \eqref{E-main-2h} due to \eqref{E-RicNs1aa},
we obtain \eqref{E-main-1i} and \eqref{E-main-2i}.
\end{proof}

\begin{example}\rm
 (i) Let $\mathcal F$ be a totally umbilical foliation (i.e., $h=\frac1n\,H\tilde g$, $T=0$). Thus,
\[
 \Phi_h=\frac{n-1}n\,{H}^\flat\otimes{H}^\flat,\quad
 \mathcal{A}^\flat = \frac{1}{n^2}\,\|H\|^2\,\tilde g,\quad
 \Psi=\frac1n\,{H}^\flat\otimes {H}^\flat,\quad
 \Sm_{\,\rm ex} = \frac{n-1}{n}\,\|H\|^2
\]
hold, and the fundamental equations \eqref{E-genRicN}$_1$ and \eqref{E-genRicN-2}$_1$ read
\begin{eqnarray}\label{E-umbRD-b}
 r_{\,\mathcal D} -\Div\tilde h -\<\tilde H, \tilde h\>
 +(\widetilde{\mathcal A} -\widetilde{\mathcal T})^\flat
 +\frac1n\,H^\flat\otimes H^\flat -{\rm Def}_{\mathcal D}\,H = 0,\\
\label{E-umbRD}
 {r}_{\,\calf} -\frac1n\,\big(\Div H +\frac{n-1}n\,\|H\|^2\big)\,\tilde g
 +\widetilde\Psi -{\rm Def}_{\calf}\,\tilde H = 0.
\end{eqnarray}

(ii) Let $\mathcal F$ be a totally geodesic foliation (i.e., $h=0=T$) of a closed Riemannian manifold $(M,g)$.
Thus, $\Sm_{\,\rm ex}=0$, and \eqref{E-genRicN}$_1$ and \eqref{E-genRicN-2}$_1$ read
\begin{equation*}
 r_{\,\mathcal D}
 -\Div\tilde h -\<\tilde H, \tilde h\> +(\widetilde{\mathcal A}+\widetilde{\mathcal T})^\flat = 0,\quad
 {r}_{\,\calf} - {\rm Def}_{\calf}\,\tilde H +\widetilde\Psi = 0.
\end{equation*}
If $g\in{\rm Riem}(M,\,{T\calf},\,{\mathcal D})$ is a critical point of $J_{\rm mix}$ with respect to $\mathcal D$-variations, then, see Euler-Lagrange equations \eqref{E-main-i-hF} and \eqref{E-main-i},
\begin{eqnarray}\label{E-main-i-hFtg}
 &&\hskip-7mm
 \Div(\tilde h -\tilde H\,g^\perp) -2\,\widetilde{\mathcal T}^\flat +\tilde{\cal K}^\flat
  = \frac12\,\big(\widetilde\Sm_{\,\rm ex}+\|\tilde T\|^2 -\Sm^*_{\,\rm mix}(M,g)\big)\,g^\perp,\\
\label{E-main-i-tg}
 &&\hskip-7mm {r}_{\,{\mathcal D}}-\<\tilde h,\tilde H\>+\widetilde{\mathcal A}^\flat-\widetilde{\mathcal T}^\flat
 +\tilde{\cal K}^\flat
 =\frac12\,\big(\Sm_{\,\rm mix} -\Sm^*_{\,\rm mix}(M,g) + \Div\tilde H \big)\,g^\perp.
\end{eqnarray}

(iii) Let $\mathcal F$ be a Riemannian foliation (i.e., $\tilde h=0=T$) of a closed manifold $(M,g)$. Note that $\widetilde\Sm_{\,\rm ex}=0$. If~$g\in{\rm Riem}(M,\,{T\calf},\,{\mathcal D})$ is a critical point of $J_{\rm mix}$ with respect to $T\calf$-variations, then, see Euler-Lagrange equations \eqref{E-main-ii-hF} and \eqref{E-main-ii},
\begin{eqnarray}\label{E-main-ii-hFth}
 \Div(h -H\,\tilde g) +\Phi_{\tilde T} \eq \frac12\,\big(\Sm_{\,\rm ex} +\|\tilde T\|^2 -\Sm^*_{\,\rm mix}(M,g)\big)\,\tilde g,\\
\label{E-main-ii-Riem}
 {r}_{\,\calf} +\Phi_{\tilde T} +\widetilde\Psi
 \eq\frac12\,\big(\Sm_{\,\rm mix} -\Sm^*_{\,\rm mix}(M,g) +\Div H \big)\,\tilde g.
\end{eqnarray}
\end{example}

Next we consider sufficient conditions for critical metrics of the action \eqref{E-Jmix}.

\begin{cor}\label{C-03} Let $\calf$ be a foliation  on a closed Riemannian manifold~$(M,g)$ with integrable normal distribution~${\mathcal D}$. Then
\begin{eqnarray}\label{E-main-i-hF-0}
 \Div(\tilde h -\tilde H\,g^\perp) +\Phi_h
 = \frac12\big(\Sm_{\,\rm ex} +\widetilde\Sm_{\,\rm ex} -\Sm^*_{\,\rm mix}(M,g)\big)\,g^\perp
 \quad ({\rm for}~{\mathcal D}{\rm-variations}), \\
 \label{E-main-ii-hF-0}
 \Div(h -H\,\tilde g) +\Phi_{\tilde h}
 = \frac12\big(\Sm_{\,\rm ex} +\widetilde\Sm_{\,\rm ex} -\Sm^*_{\,\rm mix}(M,g)\big)\,\tilde g
 \quad({\rm for}~{T\calf}{\rm-variations}),
\end{eqnarray}
where, see \eqref{E-S*main},
\begin{equation}\label{E-S*main-inegrable}
 \Sm^*_{\,\rm mix}(M,g) =\left\{\begin{array}{cc}
 \Sm_{\,\rm mix}(M,g) -\frac2p\,\Sm_{\,\rm ex}(M,g) & {\rm for}~{\mathcal D}{\rm-variations},\\
 \Sm_{\,\rm mix}(M,g) -\frac2n\,\widetilde\Sm_{\,\rm ex}(M,g) & {\rm for}~{T\calf}{\rm-variations}\,.
\end{array}\right.
\end{equation}

{\rm (i)} If $g\in{\rm Riem}(M,\,{T\calf},\,{\mathcal D})$ is a critical point of $J_{\,\rm mix}$ with respect to ${\mathcal D}$-variations
and $\calf$ is totally geodesic then $\Div(\tilde h -\frac1p\,\tilde H\,g^\perp)=0$.

{\rm (ii)} If $g\in{\rm Riem}(M,\,{T\calf},\,{\mathcal D})$ is a critical point of $J_{\,\rm mix}$ with respect to $T\calf$-variations
and $\calf$ is totally umbilical then $\Phi_{\tilde h}=\frac1n\,\widetilde\Sm_{\,\rm ex}\,\tilde g$.
\end{cor}

\begin{proof}
From \eqref{E-main-i-hF}--\,\eqref{E-main-ii-hF} with $\tilde T=0$ we obtain \eqref{E-main-i-hF-0}--\,\eqref{E-main-ii-hF-0}.
(i)~Tracing \eqref{E-main-i-hF-0} with $\Phi_h=0=\Sm_{\,\rm ex}$, we find
 $\widetilde\Sm_{\,\rm ex} -\Sm^*_{\,\rm mix}(M,g) = 2\,\frac{1-p}{p}\Div\tilde H $.
From this and \eqref{E-main-i-hF-0}  the claim follows.
(ii)~Similarly, tracing \eqref{E-main-ii-hF-0} with $h=\frac1n H\,\tilde g$, we find
 $\Sm_{\,\rm ex} +\widetilde\Sm_{\,\rm ex} -\Sm^*_{\,\rm mix}(M,g)
 =2\,\frac{1-n}{n}\Div H + \frac2n\,\widetilde\Sm_{\,\rm ex}$.
From this and \eqref{E-main-ii-hF-0}  the claim follows.
\end{proof}

\begin{rem}\rm
 By Corollary~\ref{C-03}, if
$\calf$ is a totally geodesic foliation on a closed Riemannian manifold~$(M,g)$ with integrable normal distribution ${\mathcal D}$, and $g\in{\rm Riem}(M,\,{T\calf},\,{\mathcal D})$ is critical for the action $J_{\,\rm mix}$ with respect to all adapted variations of $(M,g)$ then
\begin{equation}\label{E-scalJ-tgeo}
 \Div\Big(\tilde h-\frac1p\,\tilde H\,g^\perp\Big) = 0,\qquad
 \Phi_{\tilde h}=\frac1n\,\widetilde\Sm_{\,\rm ex}\,\tilde g\,.
\end{equation}
Obviously, \eqref{E-scalJ-tgeo}$_1$ is satisfied when $\mathcal D$ is totally umbilical,
i.e., $\tilde h=\frac1p\,\tilde H\,g^\perp$. On the other hand,
\eqref{E-scalJ-tgeo}$_2$ is not satisfied when $\mathcal D$ is totally umbilical (but not totally geodesic).

\noindent
\textbf{Problem}: \textit{Classify closed Riemannian manifolds admitting a totally geodesic foliation with integrable normal distribution satisfying $($any of\,$)$ conditions \eqref{E-scalJ-tgeo}}.
\end{rem}

\subsection{Double-twisted products of manifolds}\label{subsec:dtwisted}

Let $M = M_1\times M_2$ be a product of semi-Riemannian manifolds $(M_i\,,\,g_i)$ ($i\in\{1,2\}$).
Let $\pi_i : M\to M_i$ and $P_i = d\pi_i: TM\to TM_i$ be the canonical projections.
Given twisting functions $f_i\in C^\infty(M)$
a \textit{double-twisted product} $M_1\times_{(f_1,f_2)} M_2$ is $M_1 \times M_2$ with the metric
\begin{equation}\label{E-dtmetric}
 g =e^{f_1} \, \pi_1^\ast g_1 + e^{f_2} \, \pi_2^\ast g_2.
\end{equation}
If $f_1 = {\rm const}$ then we have a \textit{twisted product}
(a \textit{warped product} if, in addition, $f_2 = F \circ \pi_1$ for some $F \in C^\infty (M_1)$).
Both families, the \textit{leaves} $M_1\times\{y\}$ and the \textit{fibers} $\{x\}\times M_2$, are totally umbilical in $(M,g)$ and this property characterizes double-twisted products  (cf. R.\,Ponge \& H.\,Reckziegel,~\cite{pr}).
We~have $T=\tilde T=0$~and
\begin{eqnarray*}
 A_Y\eq -Y(f_1)\,\widetilde\id,\quad \ \ h=-(\nabla^\perp f_1)\,\tilde g,\quad H=-n\nabla^\perp f_1,\\
 \tilde A_X\eq -X(f_2)\id^{\!\perp} ,\quad
 \tilde h=-(\widetilde\nabla f_2)\,g^\perp,\quad
 \tilde H=-p\,\widetilde\nabla f_2,
\end{eqnarray*}
where $X\in\widetilde{\mathcal D}$ and $Y\in{\mathcal D}$ are unit vectors.
In this case, cf. \eqref{E-divN},
\begin{equation*}
 \Div\tilde H= -p\,\widetilde\Delta\,f_2 -p^2 \|\widetilde\nabla f_2\|^2,\quad
 \Div H= -n\,\Delta^\perp f_1 -n^2 \|\nabla^\perp f_1\|^2.
\end{equation*}
For any double-twisted product, we have identities, see \eqref{E-umbRD},
\begin{eqnarray}\label{E-dtp-gen}
 r_{\,{\mathcal D}} \eq \frac1p\,\Big(\Div\tilde H +\frac{p-1}{p}\,\|\tilde H\|^2\Big) g^\perp
 +{\rm Def}_{\mathcal D}\,H -\frac1n\,H^\flat\otimes H^\flat\,,\\
 \label{E-dtp-gen2}
 r_{\calf} \eq \frac1n\,\Big(\Div H +\frac{n-1}{n}\,\|H\|^2\Big) \tilde g
 +{\rm Def}_{\calf}\,\tilde H -\frac1p\,\tilde H^\flat\otimes\tilde H^\flat.
\end{eqnarray}
Hence, $\Sm_{\rm mix} =\Div(H+\tilde H)+\frac{n-1}{n}\,\|\tilde H\|^2 +\frac{p-1}{p}\,\|\tilde H\|^2$
and $J_{\,\rm mix}(g)\ge0$.

The leaves $M_1\times\{y\}$ of a twisted product are totally geodesic submanifolds on $M$ (i.e., $h=0$).
For a twisted product (i.e., $f_2=0$), we have identities, see \eqref{E-dtp-gen} and~\eqref{E-dtp-gen2},
\begin{equation*}
 r_{\,{\mathcal D}} = \frac1p\Big(\Div\tilde H +\frac{p-1}{p}\,\|\tilde H\|^2\Big)g^\perp,\qquad
  r_{\calf} = {\rm Def}_{\calf}\,\tilde H -\frac1p\,\tilde H^\flat\otimes\tilde H^\flat.
\end{equation*}

\begin{cor}\label{C-dtwist}
A double-twisted product metric \eqref{E-dtmetric} on a closed manifold $M=M_1^n\times M_2^p$ with $n,p>1$ is a critical point of the functional \eqref{E-Jmix} with respect to ${\mathcal D}$-variations
if and only if $f_1={\rm const}$ $($i.e., a twisted product$)$.
Similarly, in the case of $\,T\calf$-variations, we get $f_2={\rm const}$.
\end{cor}

\begin{proof}
Let the metric $g$ be critical with respect to ${\mathcal D}$-variations.
Since $\tilde h=\frac1p\,\tilde H\,g^\perp$ and $\Phi_h=\frac{n-1}{n}\,H^\flat\otimes H^\flat$, by \eqref{E-main-i-hF},  we have
\begin{equation}\label{E-dtw-1}
 \frac{1-p}{p}\,(\Div\tilde H)\,g^\perp +\frac{n-1}{n}\,H^\flat\otimes H^\flat
  = \frac12\,\Big(\frac{n-1}{n}\,\|H\|^2 +\frac{p-1}{p}\,\|\tilde H\|^2 -\Sm^*_{\,\rm mix}(M,g)\Big)\,g^\perp.
\end{equation}
Since the symmetric tensor $H^\flat\otimes H^\flat$ has rank $\le 1$,
then for $p>1$, we obtain $H=0$; in this case, \eqref{E-dtw-1} becomes identity.
This corresponds to totally geodesic leaves, i.e., the twisted product.
 The converse claim is also true. The case of $\,T\calf$-variations is similar.
\end{proof}

By Corollary~\ref{C-dtwist}, among double-twisted products, the (biscaling of)
direct product metrics are only critical points of the functional $J_{\,\rm mix}$ with respect to all adapted variations.

\begin{example}[Foliations of the standard sphere omitting a codimension $2$ totally geodesic submanifold]
{\rm Let $S^m = \{ x \in {\mathbb R}^{m+1} : \sum_{i=1}^{m+1} x_i^2 = 1 \}$ be the unit sphere in ${\mathbb R}^{m+1}$ and $\Sigma \subset S^m$ a codimension $2$
totally geodesic submanifold. Then (by an appropriate choice of coordinates on ${\mathbb R}^{m+1}$),
$\Sigma =\{ x \in S^m : x_1 = x_2 = 0 \}$.
\textit{
Let $S^{m-1}_+$ be the hemisphere $\{ y = (y^\prime \,  , \, y_m ) \in S^{m-1} : y_m > 0 \}$,
where $y^\prime = (y_1, \cdots, y_{m-1})$, and ${\mathcal F}_\Sigma$ be the foliation of $S^m\setminus\Sigma$,
whose leaf space is
\[
 \left( S^m \setminus \Sigma \right) /{\mathcal F}_\Sigma = \left\{ L_\zeta : \zeta \in S^1 \right\} , \;\;\;
 L_\zeta \equiv \left\{ \left( y^\prime \, , \, {\rm Re}(\zeta ) \, y_m \, , \, {\rm Im}(\zeta ) \, y_m \right) : y \in S^{m-1}_+ \right\}\,.
\]
Let $\widetilde{\mathcal D} = T\calf_\Sigma$ and $\mathcal D$ be its $g_m$-orthogonal complement
with respect to the canonical Riemannian metric $g_m \in {\rm Riem}(S^m \setminus \Sigma )$.
Then $g_m$ is a $\mathcal D$-critical point of the functional~\eqref{E-Jmix}}.

To prove the claim, note that $S^m \setminus \Sigma$ is isometric to the warped product $S^{m-1}_+ \times_w S^1$ with the warping function $w(y) = y_m$.  Precisely let $g_N$ be the first fundamental form of $S^N$ in the Euclidean space ${\mathbb R}^{N+1}$.
Then the map $I : S^{m-1}_+ \times S^1 \to S^m \setminus \Sigma$
given by
\[
 I(y \, , \, \zeta ) = \left( y^\prime \, , \, {\rm Re}(\zeta ) \, y_m \, , \, {\rm Im}(\zeta ) \, y_m \right) , \;\;\; y \in S^{m-1}_+ \, , \;\;\; \zeta \in S^1 \, ,
\]
is an isometry of $S^{m-1}_+ \times S^1$ with the warped product metric $\pi_1^\ast g_{m-1} + (w \circ \pi_1 )^2 \, \pi_2^\ast g_1$ onto $(S^m \setminus \Sigma , \, g_m )$.
Here $\pi_1 : S^{m-1}_+ \times S^1 \to S^{m-1}_+$ and $\pi_2 : S^{m-1}_+ \times S^1 \to S^1$ are the projections. Let ${\mathcal F}_+$ be the foliation of $S^{m-1}_+ \times S^1$ whose leaves are
 $(S^{m-1}_+ \times S^1 )/{\mathcal F}_+ = \{ S^{m-1}_+ \times \{ \zeta \} : \zeta \in S^1 \}$
so that $I$ is a foliated map of
 $\left(S^{m-1}\times S^1\,,\,{\mathcal F}_+\right)$ onto $\left(S^m\setminus\Sigma\,,\,{\mathcal F}_\Sigma\right)$.
Finally (by Corollary~\ref{C-dtwist}) $\pi_1^\ast \, g_{m-1} + (w \circ \pi_1 )^2 \, \pi^\ast_1 \, g_1$
is $(T\calf_+)^\bot$-critical.
}
\end{example}

\subsection{Codimension-one foliations}\label{subsec:codim1fol}

Let $\calf$ be a codimension one foliation (i.e., $p=1$) of a closed manifold~$(M^{n+1},g)$,
 and let the normal subbundle ${\mathcal D}$ be spanned by a~unit field $N\in{\mathfrak X}_M$.
We have $T=0=\tilde T$ and
\[
 h(X,Y)=g(\nabla_X\,Y,\,N),\quad
 A_N(X)=-\nabla_X\,N,\quad
 {\mathcal A}=A_N^2\quad
 (X,Y\in T\calf),
\]
where $h$ is the scalar second fundamental form and $A_N$ the Weingarten operator of ${\calf}$.
Note that $\tilde H=\nabla_{N}\,N$ is the curvature vector of $N$-curves,
$\tilde h=\tilde H g^\perp$, $\widetilde\Psi=\tilde H^\flat\otimes \tilde H^\flat$ and
\[
 \tilde A_Y(N)=g(\tilde H,Y)N,\quad
 \widetilde{\mathcal A}=\sum\nolimits_{\,a} g(\tilde H, E_a)^2 N =\|\tilde H\|^2 N\quad
 (Y\in T\calf).
\]
 We also have, see \eqref{E-Rictop2} and \eqref{E-Rictop}:
 ${r}_{\,\mathcal D} =\Ric_g(N,N)\,g^\perp$ and
 ${r}_{\,\calf} =(R_N)^\flat$\,.

Power sums of the principal curvatures $k_1, \ldots, k_n$ (the eigenvalues of $A_N$) are given~by \cite{rw-m}
\begin{equation*}
 \tau_\alpha = k_1^\alpha + \ldots + k_n^\alpha = \tr(A_N^\alpha),\quad \alpha\ge0.
\end{equation*}
The $\tau$'s can be expressed using the elementary symmetric functions of the eigenvalues of $A_N$
(called \textit{mean curvatures} in the literature),
\[
  \sigma_0 =1,\quad
  \sigma_a = \sum\nolimits_{i_1 <\ldots <i_a} k_{i_1}\cdot\ldots\cdot k_{i_a}\quad (1\le a\le n).
\]
For example, $\sigma_1=\tau_1=\tr A_N$, and $2\,\sigma_2 = \tau_1^2 - \tau_2$ when $n>1$
(and $\tau_1^2 - \tau_2=0$ when $n=1$).
The equality $\Div N=-\tau_1$ holds and $H=\tau_1 N$ is the mean curvature vector of~${\calf}$.

 Note that $\Sm_{\,\rm ex}=\tau_1^{2}-\tau_2$ and definitions \eqref{E-S*main} read
\begin{equation}\label{E-S*-1}
 \Sm^*_{\,\rm mix}(M,g) =\left\{\begin{array}{cc}
   \Sm_{\,\rm mix}(M,g) -2\,\Sm_{\,\rm ex}(M, g) & {\rm for}~{\mathcal D}{\rm-variations},\\
   \Sm_{\,\rm mix}(M,g) & {\rm for}~{T\calf}{\rm-variations}\,.
  \end{array}\right.
\end{equation}
 By \eqref{E-genRicN-2}$_1$ (see also \eqref{E-RicNs1aa}), we obtain
\begin{equation}\label{E-genRicN-p1}
 (R_N + A_N^2)^\flat = \nabla_N\,h  -\tilde H^\flat\otimes\tilde H^\flat  +{\rm Def}_{\calf}\,\tilde H.
\end{equation}
Then we find (tracing \eqref{E-genRicN-p1} or by \eqref{E-RicNs1aa}$_2$ with $T=0$) (cf. also \cite{rw-m,wa1})
\begin{equation}\label{eq-ran1}
 \Ric_g (N,N) =\Div(\tau_1 N+\tilde H) +\tau_1^{2}-\tau_2.
\end{equation}
Since $M$ is a closed manifold, by \eqref{eq-ran1} and the Divergence Theorem, we represent \eqref{E-Jmix-N}~as
\begin{equation}\label{E-Jmix3}
 J_{\,\rm mix}(g) =\int_{\,M}(\tau_1^{2}-\tau_2)
 \,{\rm d}\vol_g\,;
\end{equation}
thus, $\Sm_{\,\rm mix}(M,g)=\Sm_{\,\rm ex}(M,g)$.
 By \eqref{E-divP}, $\Div(F\,N) = \nabla_N\,F - \tau_1 F$ for $(0,2)$-tensors $F$ on $M$.

Next, we calculate the gradient of $J_{\rm mix}$ with respect to adapted variations of a metric
(see \cite[$\S$~2.3.3]{rw-m} for ${T\calf}$-variations).

\begin{cor}[of Theorem~\ref{T-main01}]\label{T-main02}
Let $\calf$ be a codimension one foliation on a closed manifold $M^{n+1}$,
whose transversal distribution ${\mathcal D}$ is spanned by a nonzero vector field $N$.
If $g\in{\rm Riem}(M,\,{T\calf},\,{\mathcal D})$ is a critical point of the functional \eqref{E-Jmix-N} with respect to variations~\eqref{E-Sdtg}, then the following Euler-Lagrange equa\-tions are satisfied:
\begin{eqnarray}\label{E-RicNs0F}
 \tau_1^{2}-\tau_2 \eq -\Sm^*_{\,\rm mix}(M,g) \quad ({\rm for}~{\mathcal D}{\rm-variations}),\\
\label{E-RicNs1F}
 \Div\big((h -\tau_1\,\tilde g)N\big) \eq \frac12\,\big(\tau_1^{2}-\tau_2 -\Sm^*_{\,\rm mix}(M,g)\big)\,\tilde g
 \quad ({\rm for}~{T\calf}{\rm-variations}).
\end{eqnarray}
\end{cor}

\begin{proof}
From \eqref{E-main-2h} and \eqref{E-main-1h} with $p\leftrightarrow n$ and then using $T=0$ we have \eqref{E-RicNs0F} and \eqref{E-RicNs1F}.
\end{proof}

\begin{rem}\label{R-03F}\rm
(i) By \eqref{E-RicNs0F}, if $n=1$ then $J_{\,\rm mix}(g)=0$, and
if $n>1$, then $\sigma_2 = {\rm const}$,
resp., $\sigma_2 =0$ when variations do not preserve the volume of $(M,g)$.
 Replacing terms in \eqref{E-RicNs0F}--\,\eqref{E-RicNs1F} due to \eqref{E-genRicN-p1}--\,\eqref{eq-ran1}
(or using \eqref{E-main-1i} and \eqref{E-main-2i} with $\tilde T=0$),
we rewrite the Euler-Lagrange equations~as
\begin{eqnarray}\label{E-main-i2}
 &&\hskip-13mm\Ric_g(N,N) =
 \frac12\big(\Ric_g(N,N) -\Sm^*_{\,\rm mix}(M,g) +\Div(\tau_1 N + \tilde H) \big)
 \quad ({\rm for}~{\mathcal D}{\rm-variations}),\\
\label{E-main-ii1}
 &&\hskip-13mm (R_N +A_N^2)^\flat-\tau_1 h +\tilde H^\flat\otimes \tilde H^\flat -{\rm Def}_{\calf}\,\tilde H \\
\nonumber
 && =\frac12\,\big(\Ric_g(N,N) - \Sm_{\,\rm mix}(M,g) +\Div(\tau_1 N-\tilde H) \big)\,\tilde g
 \quad ({\rm for}~T\calf{\rm-variations}).
\end{eqnarray}

(ii) Equation \eqref{E-RicNs1F} for $T\calf$-variations \eqref{E-Sdtg}--\,\eqref{E-Sdtg-2} is equivalent to \cite[Example~2.5]{rw-m}, where notation $J_{\,\rm mix}=2\,E_N$ is used and the Euler-Lagrange equations
are given in the form
\begin{equation}\label{E-RicNs222}
 -\Div(T_1(h)N) = \frac12\,\left( \tau_1^{2}-\tau_2 - \Sm_{\,\rm mix}(M,g)\right)\,\tilde g.
\end{equation}
Here, $T_1(h)=\tau_1\tilde g-h$ is the first Newton transformation of $h$.
By the above,
\[
 -\Div(T_1(h)N) = \nabla_N\,h -\tau_1 h -\Div(\tau_1 N)\,\tilde g;
\]
hence, \eqref{E-RicNs222} reduces to \eqref{E-RicNs1F}.
\end{rem}

Since $\Phi_{\,\tilde h}=0=\widetilde\Sm_{\,\rm ex}$, from Proposition~\ref{C-03}(ii) we have the following.

\begin{prop}\label{C-2-9}
Let $\calf$ be a codimension-one foliation of a~closed Riemannian manifold $(M^{n+1},g)$ $(n>1)$
with normal distribution spanned by a unit vector field $N$.
Then $g$ is a critical point of $J_{\,\rm mix}$ with respect to adapted variations
preserving the volume of $(M,g)$ if and only if
\begin{equation}\label{E-Flow3}
 \tau_1=0,\quad \sigma_2={\rm const}\le0,\quad \nabla_N\,h=0.
\end{equation}
In particular, \eqref{E-Flow3} are satisfied trivially when $\calf$ is totally geodesic.
\end{prop}

\begin{proof}
By \eqref{E-RicNs0F} and \eqref{E-S*-1}, we have $\Sm_{\,\rm ex} =\Sm_{\,\rm mix}(M,g)$,
hence $\sigma_2=-\tau_2={\rm const}\le0$.
Tracing \eqref{E-RicNs1F} and using \eqref{E-S*-1}, we obtain
\[
 (1-n)\big(N(\tau_1)-\tau_1^2\big) = \frac n2\,(\tau_1^{2}-\tau_2 -\Sm_{\,\rm mix}(M,g))
\]
Since $n>1$, we obtain the ODE $N(\tau_1)=\tau_1^2$ on complete $N$-curves,
which has trivial solution $\tau_1=0$. Then, by \eqref{E-RicNs1F}, we have $\nabla_N\,h=0$.
\end{proof}

\begin{cor}\label{C-dim3}
Let $\calf$ be a 2-dimensional foliation on a closed 3-dimensional Riemannian manifold $(M,g)$
with normal distribution spanned by a unit vector field $N$.
Then $g$ is a critical point of $J_{\,\rm mix}$ with respect to adapted variations
preserving the volume of $(M,g)$ if and only if

$1)$ the principal curvatures $k_1,k_2$ of the leaves are constant on $M$ and $k_1+k_2=0$;

$2)$ the corresponding eigenvectors $E_1$ and $E_2$ are parallel in the $N$-direction.

\noindent
In particular, $M$ is parallelizable: $\{E_1,E_2,N\}$ is the global orthonormal frame.
\end{cor}

\begin{example}[Foliations by level hypersurfaces]\label{ex:6}\rm
Let $(M^{n+1},g)$ be a Riemannian manifold, $u \in C^\infty (M, {\mathbb R})$, and let $\mathcal F$ be the foliation
of $U = M \setminus {\rm Crit}(u)$ by level hypersurfaces of $u$.
Then $N =\lambda^{-1}\,\nabla u$, where $\lambda\equiv\|\nabla u\|$ and
 $\tau_1 = -\Div N = -\lambda^{-1}\Div(\nabla u) -(\nabla u)(\lambda^{-1})$,~i.e.,
\begin{equation}\label{e:S.81}
 \tau_1 = -\lambda^{-1}\,\Delta u + \lambda^{-2}\,(\nabla u)(\lambda ).
\end{equation}
The geometric background needed to write the Euler-Lagrange equations
of Corollary~\ref{T-main02} is presented in \cite[p. 104--116]{Ton}.
We need to rephrase a few facts in \cite{Ton} under the conventions adopted in this paper.
Let $\pi:TU\to\nu({\mathcal F}) \equiv TU/T\calf$ be the transverse bundle.
Let $h$ be the second fundamental form of $\mathcal F$ in $(U, g)$, i.e.,
$h(X,Y) = \pi(\nabla_X Y)$ for any $X,Y \in\mathfrak{X}_{\widetilde{\mathcal D}}$.
Then
\begin{equation}\label{e:S.82}
 h(X,Y) = - \lambda^{-1}  \, {\rm Hess}_u (X,Y) \, \pi(N)
\end{equation}
by \cite[(8.5),~p.~105]{Ton}.
Here ${\rm Hess}_u$ is the Hessian of $u$, i.e., ${\rm Hess}_u (V,W)=(\nabla_V \,{\rm d} u ) W$ for any $V,W\in\mathfrak{X}_{U}$.
Let $\sigma : \nu({\mathcal F})\to T^\bot\calf \equiv {\mathcal D}$ be the natural bundle isomorphism,
i.e., $\sigma(s) = V^\bot$ for any $s\in\nu({\mathcal F})$ and any $V \in\mathfrak{X}_{U}$ such that $\pi(V) = s$.
Here $V^\bot$ is the ${\mathcal D}$-component of $V$ with respect to the decomposition $TU = T\calf \oplus {\mathcal D}$. The metric $g_Q$ induced by $g$
on the transverse bundle is $g_Q (r, s) = g(\sigma (r) , \sigma (s))$ for any $r,s \in Q \equiv \nu ({\mathcal F})$.
The Weingarten operator $A_N : T\calf \to T\calf$ is given~by
\[
 g(A_N(X)\, ,\, Y) = g_Q (h(X,Y)\,,\,\pi(N)).
\]
Then (by (\ref{e:S.82}))
\begin{equation}\label{e:S.83}
 g(A_N X, Y) = - \lambda^{-1}\, {\rm Hess}_u (X,Y).
\end{equation}
Let $\nu = \lambda^{-1} \,{\rm d} u$ be the transverse volume element. Let $\{ E_a : 1 \leq a \leq n \}$ be a local orthonormal frame of $T\calf$ and
$\{ \omega^a : 1 \leq a \leq n \}$ the dual coframe, i.e., $\omega^i (E_j ) = \delta^i_j$ and $\omega^i (N) = 0$.
Then $\{ E_a \, , \, N : 1 \leq a \leq n \}$ is a local orthonormal frame of $TU$
and $\{ \omega^a \, , \, \nu : 1 \leq a \leq n \}$ is the corresponding dual coframe,
and (\ref{e:S.83}) may be locally written
\begin{equation}\label{e:S.84}
 A_N X = - \lambda^{-1} \sum\nolimits_{a} {\rm Hess}_u (X, E_a )\, E_a \,.
\end{equation}
Then (by (\ref{e:S.84}))
\[
 \tau_2 = {\rm Tr} \left( A_N^2 \right) = \sum\nolimits_{a} g(A_N E_a \, , \, A_N E_a )
 = \lambda^{-2} \sum\nolimits_{a,b} ({\rm Hess}_u (E_a \, , \, E_b))^2,
\]
so that (by ${\rm Hess}_u (N,N) = N^2 (u)$)
\begin{equation}\label{e:S.85}
 \tau_2 = \lambda^{-2} \left\{ \left\| {\rm Hess}_u \right\|^2 - (N^2(u))^2 \right\} .
\end{equation}
Note that $N(u) = \lambda$; hence,
\begin{equation}\label{e:S.86}
 N^2 (u) = \frac{1}{\| \nabla u \|} \, g \left( \nabla u \, , \, \nabla (\| \nabla u \| ) \right) .
\end{equation}
Then (by (\ref{E-Jmix3}), (\ref{e:S.81}) and (\ref{e:S.85})-(\ref{e:S.86}))
\begin{eqnarray}\label{e:S.87}
\nonumber
 J_{\,\rm mix} (g) \eq \int_{M} \Big\{ \frac{1}{\| \nabla u \|^2}
 \,\Big(\Delta u -\frac{1}{\|\nabla u \|} \,g\left(\nabla f\, ,\,\nabla (\| \nabla f \| ) \right) \Big)^2  \\
 \minus \frac{1}{\| \nabla u \|^2} \, \Big( \left\| {\rm Hess}_u \right\|^2 - \frac{1}{\| \nabla u \|} \, g \left( \nabla u \, , \, \nabla (\| \nabla u \| ) \right) \Big)  \Big\} \, d \, {\rm vol}_g \, .
\end{eqnarray}
By \eqref{E-RicNs0F}, a metric $g$ on $U$ is critical for (\ref{e:S.87}) with respect to $\mathcal D$-variations if
$\tau_1^2-\tau_2=0$, that is trivial for $n=1$, or (by (\ref{e:S.81}) and (\ref{e:S.85}))
 $\left(  \Delta u - \lambda^{-1} \, (\nabla u)(\lambda ) \right)^2
 = \left\| {\rm Hess}_u \right\|^2 - N^2 (u)$
that is (by~(\ref{e:S.86}))
\begin{equation}\label{e:S.88}
 \Big( \Delta u - \frac{(\nabla u )(\| \nabla u \| )}{\| \nabla u \|} \Big)^2
 =\left\| {\rm Hess}_u \right\|^2
 - \Big(\frac{(\nabla u )(\| \nabla u \| )}{\| \nabla u \|}\Big)^2 \, .
\end{equation}
\end{example}

\begin{example}[Foliations in thermodynamics]\rm
\,Let $M \subset {\mathbb R}^k$ be the phase space of a thermodynamical system.
Let $P, \, V, \, E \in C^\infty (M, {\mathbb R})$ be respectively the pressure, volume and internal energy. If $\omega = d E + P \, d V \in \Omega^1 (M)$ then $\omega \wedge d \omega = 0$ by the second principle of thermodynamics (cf. e.g. \cite{Car}), i.e. the distribution $\tilde{\mathcal D} = {\rm Ker} (\omega )$ is involutive. Let $\mathcal F$
be the foliation of $M$ such that $T({\mathcal F}) = \tilde{\mathcal D}$. Let $g \in {\rm Riem}(M)$.
If $\lambda = \| \omega \|$ then $N = \lambda^{-1} \left( \nabla E + P \, \nabla V \right)$.  Next
\[
 A_N  = - \frac{1}{\lambda} \left( \nabla \nabla E + (d P ) \otimes \nabla V + P \nabla \nabla V \right) + \frac{1}{\lambda^2}\,(d \lambda )\otimes \left( \nabla E + P \nabla V\right)
\]
so that
\begin{eqnarray*}
 \tau_1 = - \lambda^{-1} \left( \Delta E + (\nabla V)(P) + P \, \Delta V \right)
 +\lambda^{-2} \left( (\nabla E)(\lambda ) + P (\nabla V) (\lambda ) \right) ,\\
 \tau_2 = \lambda^{-2} \left( \| {\mathcal E} \|^2 + 2 P \, g^\ast ({\mathcal E} , {\mathcal V}) + P^2 \, \| {\mathcal V} \|^2 + \| \nabla V \|^2 \, \| \nabla P \|^2 \right) +\\
 + 2 \lambda^{-2} \,\left( g(\nabla_{\nabla P} \nabla E , \nabla V )  + P \, g(\nabla_{\nabla P} \nabla V \, , \, \nabla V ) \right) +\\
 + \lambda^{-4}\, \| \nabla \lambda \|^2 \left( \|\nabla E \|^2 + 2 P \, g(\nabla E , \nabla V) + P^2 \, \| \nabla V \|^2 \right),
\end{eqnarray*}
where we have set ${\mathcal V} X = \nabla_X \nabla V$ and ${\mathcal E} X = \nabla_X \nabla E$ for every $X \in \mathfrak{X}_{\mathcal F}$.  All processes in a simple mechanically isolated thermodynamic system are isochoric,
i.e., $V =$ const. If this is the case $\mathcal F$ is the foliation of $M \setminus {\rm Crit}(E)$ by surfaces
of constant internal energy, a situation covered by Example~\ref{ex:6}, i.e. a metric $g$ is critical with respect to ${\mathcal D}$-variations if $\Delta E = f_\pm (E)$, where
$f_\pm (E) = \| \nabla E \|^{-1} (\nabla E)(\| \nabla E \| )  \pm \left[ \| \nabla \nabla E \|^2 + \| \nabla E \|^{-2} \| \nabla (\| \nabla E \| ) \|^2 \right]^{\frac{1}{2}}$.
\end{example}

\begin{example}[Foliations by level lines on open Riemann surfaces]\rm
Let $M$ be a closed surface, then $J_{\,\rm mix}(g) = 2\pi\chi(M)$ for any $g \in {\rm Riem}(M)$.
Hence, $D J_{\,\rm mix}(g) = 0$ (any metric is critical). However, if $M$ is an open Riemann surface,
the problem of looking for stationary metrics is well posed and nontrivial. Here, one works with a known functional, because $\Sm_{\, \rm mix}(g)$ is merely the Gaussian curvature of the metric $g$, as observed in the introduction to this paper, yet its domain is the variety ${\rm Riem}(M,\,\widetilde{\mathcal D},\,{\mathcal D})$ depending
on the given line fields $(\widetilde{\mathcal D},\,{\mathcal D})$ on $M$.
Cohn-Vossen first proved that the total curvature of a finitely connected complete noncompact
Riemannian surface $M$ is bounded above by $2\pi\chi(M)$.
The total curvature of such a manifold $M$ is not a topological invariant but is dependent
upon the choice of a Riemannian metric, see \cite{sst}.

The results in Example~\ref{ex:6} apply to foliations by level lines on real surfaces
(as studied by R. Jerrard \& L. Rubel, \cite{JeRu}) associated to harmonic functions
(e.g. real parts of holomorphic functions, cf. also \cite{Jer}).
Let $D\subset{\mathbb C}$ be a relatively compact domain and $f = u + i v$ a holomorphic function,
where $u,v : D\to{\mathbb R}$ its real and imaginary parts.
Let $\calf$ be the foliation of $U = D\setminus{\rm Crit}(u)$ by level lines of $u$
and $\mathcal D$ is spanned by $\nabla u$.

If, for instance, $f(z)= -i\,z$ with $z\in{\mathbb C}$ then, for an arbitrary metric $g\in{\rm Riem}({\mathbb C})$,
\[
 \nabla u = g^{i2} \, \frac{\partial}{\partial x^i}\, , \;\;\; \lambda = \sqrt{g^{22}}\,,\;\;\;
 (\nabla u)(\lambda) =\frac{1}{2\lambda}\big(\,g^{12}\,\left(g^{22}\right)_x + g^{22}\,\left(g^{22}\right)_y \big),
\]
\[
 \Delta u = \frac{1}{\sqrt{G}} \, \frac{\partial}{\partial x^i} \big( \sqrt{G} \, g^{i2} \big),
 \;\;\; {\rm Hess}_u \Big( \frac{\partial}{\partial x^i}\,,\,\frac{\partial}{\partial x^j} \Big)
 = - \Big\{ \begin{array}{c} 2 \\ ij \end{array} \Big\}\,,
\]
\[
 \left\| {\rm Hess}_u \right\|^2 = g^{ij} g^{k\ell} \Big\{ \begin{array}{c} 2 \\
 ik \end{array} \Big\} \Big\{ \begin{array}{c} 2 \\ j\ell \end{array} \Big\}  \, ,
\]
\[
 \Sm_{\rm mix} = K
 =\frac{1}{\lambda^2} \Big\{ \frac{1}{\sqrt{G}} \, \frac{\partial}{\partial x^i} \big(\sqrt{G}\, g^{i2}\big)
 -\frac{1}{2 \lambda^2} \big( g^{12} \left( g^{22} \right)_x + g^{22} \left( g^{22} \right)_y \big) \Big\}^2 -
\]
\[
 -\frac{1}{\lambda^2} \Big\{ g^{ij} g^{k\ell} \Big\{ \begin{array}{c} 2 \\ ik \end{array} \Big\}
 \Big\{ \begin{array}{c} 2 \\ j \ell \end{array} \Big\} -
\frac{1}{2 \lambda^2} \big( g^{12} \left( g^{22} \right)_x + g^{22} \left( g^{22} \right)_y \big) \Big\} ,
\]
where $G = \det \left[ g_{ij} \right]$. Let $\Omega \subset {\mathbb C}$.
By Corollary~\ref{T-main02}, an adapted metric $g\in{\rm Riem}({\mathbb C})$
with $K(g) = 0$ (that implies $K(M, g) = 0$) is $\mathcal D$-critical (where ${\mathcal D} = {\mathbb R}\nabla u$).

Euclidean metric $g_0 = {\rm d x}^2 + {\rm d y}^2$ is trivially a $\mathcal D$-critical point. Let us look for critical metrics
conformal to Euclidean metric, i.e., $g = e^\phi g_0$ with $\phi \in C^\infty ({\mathbb C})$. Then
\[
 g^{ij} = e^{-\phi} \, \delta^{ij} \, , \;\;\;  G = e^{2 \phi} \, , \;\;\; \lambda = e^{- \phi /2} \, , \;\;\;
 \Big\{\begin{array}{c} i \\ jk \end{array} \Big\} = \frac{1}{2} \left( \phi_{|k} \, \delta^i_j + \phi_{|j} \, \delta^i_k - \delta_{jk} \, \phi_{|i} \right) ,
\]
\[
 \frac{1}{\sqrt{G}} \, \frac{\partial}{\partial x^i} \big( \sqrt{G} \, g^{i2} \big) = 0, \;\;\;
 g^{ij} g^{k\ell} \Big\{ \begin{array}{c} 2 \\ ik \end{array} \Big\} \Big\{ \begin{array}{c} 2 \\ j \ell \end{array} \Big\} = \frac{1}{2} \, e^{-2 \phi} \left( \phi_x^2 + \phi_y^2 \right) ,
\]
\[
 g^{12} \left( g^{22} \right)_x + g^{22} \left( g^{22} \right)_y = - e^{- 2 \phi} \phi_y,\quad
 K(e^\phi g_0 ) = - \frac{1}{2} \, e^{-\phi}
 \Big( \phi_x^2 + \frac{1}{2} \, \phi_y^2 + e^\phi \phi_y \Big);
\]
hence, $g$ is $\mathcal D$-critical when $\phi$ is a solution to the PDE
\begin{equation}\label{e:S.96}
 \phi_x^2 + \phi_y^2/2 + e^\phi \phi_y = 0.
\end{equation}
For instance, the nonconstant solutions $\phi$  to (\ref{e:S.96}) with $\phi_x = 0$ are $\phi (y) = - \log |2 y + c|$, $c \in {\mathbb R}$. Consequently, if $\mathcal F$ is the model foliation by lines parallel to the $x$-axis, the adapted metrics
 $g_c = (2 y + c)^{-1} \, \left( {\rm d x}^2 + {\rm d y}^2 \right),\ c\in{\mathbb R}$,
are $\mathcal D$-critical on the half-plane $U_c \equiv \{ (x,y) \in {\mathbb R}^2 : y > - c/2 \}$.
Similarly if $f(z) = z$ (producing the foliation of ${\mathbb R}^2$ by lines parallel to the
$y$-axis) the metrics $g = (2 x + c)^{-1} \, \{ {\rm d x}^2 + {\rm d y}^2 \}$ are $\mathcal D$-critical.
\par
Let us go back to foliations $\mathcal F$
of $U = D \setminus {\rm Crit}(u)$ by level lines of $u$. The~unit normal to $\mathcal F$
in Euclidean plane $({\mathbb R}^2 , g_0 )$ is $N = \lambda^{-1} ( u_x \, \partial_x + u_y \, \partial_y )$ with $\lambda = ( u_x^2 + u_y^2 )^{\frac{1}{2}}$.
Let $\kappa\in T^1_0 U$ be given by $X \, \rfloor \, \kappa = 0$ and $\kappa (Z) = {\rm Tr}(A_Z)$ for any $X \in\mathfrak{X}_{\widetilde{\mathcal D}}$ and $Z \in {\mathfrak X}_{\mathcal D}$. Then $\tau_1 = \kappa (N)$. By a result in \cite{JeRu} (cf. also \cite[p.~112]{Ton})
\begin{equation}\label{e:S.89}
\tau_1 = \lambda^{-3}\,{\mathcal L}(u),
\end{equation}
where
 ${\mathcal L}: u \to \left( u_x^2 - u_y^2 \right) u_{xx} + 2 \, u_x \, u_y \, u_{xy}$
is the nonlinear operator.
Moreover,
$\left\|{\rm Hess}_u\right\|^2 = 2\left( u_{xx}^2 + u_{xy}^2 \right)$ and $N^2(u) = \lambda^{-2}\, {\mathcal L}(u)$;
hence (by (\ref{e:S.85})),
\begin{equation}
\tau_2 =  2 \lambda^{-2} \left( u_{xx}^2 + u_{xy}^2 \right) - \lambda^{-4} \, {\mathcal L}(u) .
\label{e:S.90}
\end{equation}
Then (by (\ref{e:S.89})-(\ref{e:S.90})) the Euclidean metric ${\rm d x}^2 + {\rm d y}^2$ is ${\mathbb R} \nabla u$-critical if
\begin{equation}\label{e:S.91}
- 2 \left( u_{xx}^2 + u_{xy}^2 \right) + \lambda^{-2} \, {\mathcal L}(u) + \lambda^{-4} \, {\mathcal L}(u)^2 = 0.
\end{equation}
If, for instance, $f(z)=z^2$ then $U={\mathbb C}\setminus\{0\}$ and $\lambda=2 r$ with $r=(x^2+y^2 )^{\frac{1}{2}}$. Also ${\mathcal L}(u) = 8 u$ and an inspection of (\ref{e:S.91}) shows that Euclidean metric is not ${\mathbb R} \nabla u$-critical.
\end{example}

The reader can find more examples for codimension one foliations in \cite{bdrs}.

\end{document}